\documentclass[12pt,a4paper]{amsart}
\usepackage[latin1]{inputenc}
\usepackage{amsmath}
\usepackage{amsthm}
\usepackage{amsfonts}
\usepackage{amssymb}
\usepackage{graphicx,caption}
\usepackage{subcaption}
\usepackage[table, dvipsnames]{xcolor}% for coloring cells on tables.
\usepackage{chngpage} % for centering oversize tables
\usepackage{color}
\usepackage{tikz}
\usetikzlibrary{arrows.meta, decorations.pathreplacing}
\usetikzlibrary{3d}

\usetikzlibrary{matrix}
\usepackage{hyperref}
\usepackage{fullpage}
\usepackage{array}
\usepackage{fullpage}
\usepackage{hhline}
\usepackage{booktabs,siunitx}
\usepackage{makecell}
\usepackage{mathtools}

\usepackage{comment}

\newcommand\edge{\mathrel{\boldsymbol{\cdot}\mkern-4mu{-}\mkern-4mu\boldsymbol{\cdot}}}
\newtheorem{theorem}{Theorem}[section]
\newtheorem{lemma}[theorem]{Lemma}

\newtheorem{proposition}[theorem]{Proposition}

\newtheorem{corollary}[theorem]{Corollary}

\theoremstyle{definition}

\newtheorem{question}[theorem]{Question}

\newcommand{\N}{\ensuremath{\mathbb{N}}}
\newcommand{\R}{\ensuremath{\mathbb{R}}}
\newcommand{\T}{\ensuremath{\mathbb{T}}}
\newcommand{\Z}{\ensuremath{\mathbb{Z}}}
\newcommand{\cU}{\ensuremath{\mathcal{U}}}

\newcommand{\vr}[2]{\mathrm{VR}(#1;#2)}

\newcommand{\cl}{\ensuremath{\mathrm{Cl}}} 
\newcommand{\diam}{\ensuremath{\mathrm{diam}}} 

\newcommand\dv{{:}} %\dv stands for "dividier"

\newcommand{\note}[1]{\textcolor{blue}{({#1})}}
\numberwithin{equation}{section}
\usepackage[left=2.5cm,right=2.5cm,top=2.5cm,bottom=2cm]{geometry}

\newcolumntype{?}{!{\vrule width 1.7pt}}

\title{Vietoris--Rips complexes of torus grids}

\author{Henry Adams}
\address{Department of Mathematics\\University of Florida\\Gainesville, FL~32611\\USA}
\email{henry.adams@ufl.edu}
\author{Adenike Yeside Adetowubo}
\address{School of Mathematical Sciences\\Hebei Normal University\\Shijiazhuang\\Hebei, 050024\\China}
\email{adetowubonikky@yahoo.com}
\author{Hector Barriga-Acosta}
\address{I Am Legacy\\Charlotte, NC\\USA}
\email{hector@iam-legacy.com}
\author{Ziqin Feng}
\address{Department of Mathematics and Statistics\\Auburn University\\Auburn, AL~36849\\USA}
\email{zzf0006@auburn.edu}
\author{John Sterling}
\address{Department of Mathematics and Statistics\\Auburn University\\Auburn, AL~36849\\USA}
\email{jss0122@auburn.edu}


\begin{document}

\begin{abstract}
We study the topology of Vietoris--Rips complexes of finite grids on the torus.
Let $T_{n,n}$ be the grid of $n\times n$ points on the flat torus $S^1\times S^1$, equipped with the $l^1$ metric.
Let $\vr{T_{n,n}}{k}$ be the Vietoris--Rips simplicial complex of this torus grid at scale $k\ge 0$.
For $n\ge 7$ and small scales $2\le k\le \frac{n-1}{3}$, the complex $\vr{T_{n,n}}{k}$ is homotopy equivalent to the torus.
For large scales $k\ge 2\lfloor\frac{n}{2}\rfloor$, the complex $\vr{T_{n,n}}{k}$ is a simplex and hence contractible.
Interesting topology arises over intermediate scales $\frac{n-1}{3}<k<2\lfloor\frac{n}{2}\rfloor$.
For example, we prove that $\vr{T_{2n,2n}}{2n-1}\cong S^{2n^2-1}$ for $n\ge 2$, that $\vr{T_{3n,3n}}{n}\simeq\vee^{6n^2-1}S^2$ for $n\ge 2$, and that $\vr{T_{3n-1,3n-1}}{n}\simeq \bigvee_{6n-3} S^2\vee \bigvee_{6n-2}S^3$ for $n\geq 3$.
Based on homology computations, we conjecture that $\vr{T_{n,n}}{k}$ is homotopy equivalent to a 3-sphere for a countable family of $(n,k)$ pairs, and we prove this for $(n,k)=(7,4)$.
\end{abstract}

\maketitle

%\let
%\cleardoublepage
%\clearpage
%\linespread{1.6}
%\setlength{\baselineskip}{1.6\baselineskip}
\pagenumbering{roman}
\pagenumbering{arabic}
%\tableofcontents

\section{Introduction}
\label{sec:intro}

The Vietoris--Rips complex originates in the work of Leopold Vietoris~\cite{Vietoris27}, a pioneer in the study of homology.
In 1927, he introduced Vietoris--Rips complexes to construct a canonical simplicial complex on top of any metric space.
Later, Elihayu Rips and Mikhail Gromov~\cite{Gromov1987} used Vietoris--Rips complexes to study hyperbolic groups in geometric group theory.
More recently, Vietoris--Rips complexes have been used in computational algebraic topology and topological data analysis to process point cloud data~\cite{Carlsson2009,EdelsbrunnerHarer,edelsbrunner2013persistent,dey2022computational,zomorodian2009computational}, based on the existence of fast algorithms to compute persistent homology~\cite{bauer2021ripser,zomorodian2010fast,zomorodian2005computing}.
The study of Vietoris--Rips complexes is finding connections to many areas of geometry and topology.

For any metric space $X$ and a scale parameter $r\ge 0$, the Vietoris--Rips complex $\vr{X}{r}$ is a simplicial complex on the vertex set $X$, which has as its simplices the finite subsets of $X$ of diameter at most $r$.
The Vietoris--Rips complex is an example of a clique complex because it is the maximal simplicial complex determined by its 1-skeleton.
Part of the applicability of Vietoris--Rips complexes comes from the work of Hausmann~\cite{Hausmann1995} and Latschev~\cite{Latschev2001}.
Hausmann proved that if $M$ is a Riemann manifold and the scale $r$ is sufficiently small, then $\vr{M}{r}$ is homotopy equivalent to $M$.
Latschev also proved that if $X$ is Gromov--Hausdorff close to $M$ and $r$ is in the correct range, then $\vr{X}{r}$ remains homotopy equivalent to $M$.

Let us now focus attention on the $m$-dimensional torus $\T^m = (S^1)^m = %\underbrace{
S^1 \times \ldots \times S^1%}_{m\text{ times}}
$.
There are a variety of metrics that one can put on the torus, including the $l^p$ metric for any $1\le p\le \infty$.
The most natural metric, in our opinion, is the $l^2$ metric; indeed, $(\T^m,l^2)$ is the flat torus, a Riemannian manifold with constant curvature zero.
Question~2 in Section~2 of~\cite{gasarch2017open} asks about the homotopy types of Vietoris--Rips complexes of the flat torus $(\T^m,l^2)$.
To our knowledge, the only known result is that $\vr{(\T^m,l^2)}{r}$ is homotopy equivalent to the torus $\T^m$ for all scale parameters $r$ that are sufficiently small, by Hausmann's theorem~\cite{Hausmann1995}.

When the $m$-dimensional torus is equipped with the $l^\infty$ metric, and when each geodesic circle has unit circumference, it is known from~\cite[Proposition~10.2]{AA-VRS1} that $\vr{(\T^m,l^\infty)}{r}$ is a product of odd-dimensional spheres, namely
\[\vr{(\T^m,l^\infty)}{r} \simeq \underbrace{\vr{S^1}{r} \times \ldots \times \vr{S^1}{r}}_{m\text{ times}} \simeq \underbrace{S^{2k+1} \times \ldots \times S^{2k+1}}_{m\text{ times}}\]
for all $\frac{k}{2k+1}<r<\frac{k+1}{2k+3}$.
An analogous formula is true if the different circular factors forming the $m$-dimensional torus have different circumferences; in this case the Vietoris--Rips complex of the $m$-torus with the $l^\infty$ metric could be homotopy equivalent to a product of odd-dimensional spheres of different dimensions.

In this paper, we instead consider the $l^1$ metric.
Part of our motivation for considering the $l^1$ metric is the following.
Though we think of the $l^1$ and $l^\infty$ metrics on the torus $\T^2$ as being less natural than the $l^2$ metric, if we can understand the inclusion map $\vr{(\T^2,l^1)}{r}\hookrightarrow \vr{(\T^2,l^\infty)}{r}$, then since this inclusion factors as
\[\vr{(\T^2,l^1)}{r}\hookrightarrow \vr{(\T^2,l^2)}{r}\hookrightarrow \vr{(\T^2,l^\infty)}{r},\]
we can deduce properties of the intermediate space $\vr{(\T^2,l^2)}{r}$.
For example, if we can show that the inclusion map $\vr{(\T^2,l^1)}{r}\hookrightarrow \vr{(\T^2,l^\infty)}{r}$ is nonzero on the $k$-th homotopy group $\pi_k$, then it follows that $\pi_k(\vr{(\T^2,l^\infty)}{r})\neq 0$, and similarly for homology groups.\footnote{See the poster~\cite{coyle2025vietoris} for some initial computations of the homology of grid graphs in the torus with the $l^2$ metric.}
A second motivation for considering the $l^1$ metric is that it is the natural (shortest path) metric to put on the grid graphs $T_{n,n}$, which approximate the torus $\T^2$, as we describe in the next paragraph.
In this paper, we also focus on the $2$-dimensional torus $\T^2 = S^1 \times S^1$, as this case is already complicated and points to phenomena that we expect to appear again with the $m$-dimensional torus $\T^m$.

We study the homotopy types and homology groups of Vietoris--Rips complexes of a finite grid of points on the torus.
Let $T_{n,n}$ be the $n\times n$ grid graph on a torus $T_{n,n}\subseteq S^{1}\times S^{1}$.
In other words, $T_{n,n}$ is the product of the cyclic graph $C_{n}$ with itself.
We consider the homotopy types and homology groups of its Vietoris--Rips complexes $\vr{T_{n,n}}{k}$ at a growing scale parameter $k\ge 0$.
Here, the input metric space into the Vietoris--Rips construction is the vertex set of the graph $T_{n,n}$, equipped with the $l^1$ metric.
We think of the $l^1$ metric as a natural metric on the vertex set of the grid graph $T_{n,n}$: it is the same as the shortest path metric on the natural grid graph with $T_{n,n}$ as its vertex set.
We note that as $n\to \infty$, the persistent homology of $\vr{T_{n,n}}{\frac{\bullet}{n}}$ converges to the persistent homology of $\vr{(\T^2,l^1)}{\bullet}$ by the stability of persistent homology~\cite{ChazalDeSilvaOudot2014}, where $(\T^2,l^1)$ is the $l^1$ metric product of two geodesic circles each of circumference one.

Table~\ref{table:Torus-Grids} shows known homotopy types and computed homology groups of the Vietoris--Rips complexes $\vr{T_{n,n}}{k}$ of $n\times n$ grids on the torus.
We write $\T^2$ in Table~\ref{table:Torus-Grids} to mean that the complex has homotopy type of the torus.
We write $\beta_i=d$ as shorthand for $H_i(\vr{T_{n,n}}{k};\Z/2)\cong (\Z/2)^d$.
The homology computations, performed using $\Z/2$ coefficients in Ripser~\cite{bauer2021ripser}, were carried out on Auburn University's Easley cluster.

\renewcommand{\arraystretch}{1.13}
\begin{table}[h]
\setlength\extrarowheight{0.6pt}
\begin{center}
\begin{adjustwidth}{-.58in}{-.5in} 
\scalebox{0.70}{
\begin{tabular}{? >{$} c <{$} ? >{$} c <{$} | >{$} c <{$} | >{$} c <{$} | >{$} c <{$} | >{$} c <{$} | >{$} c <{$} | >{$} c <{$} | >{$} c <{$} | >{$} c <{$} ?}
\Xhline{1.7pt}
%\cl(T_{n,n}^k) 
& k=1& 2 & 3 & 4 & 5 & 6 & 7 & 8 & 9 \\
\Xhline{1.7pt}
n=3 & \vee^4 S^1 & \cellcolor{lightgray!50} * & \cellcolor{lightgray!50} * & \cellcolor{lightgray!50} * & \cellcolor{lightgray!50} * & \cellcolor{lightgray!50} * & \cellcolor{lightgray!50} * & \cellcolor{lightgray!50} * & \cellcolor{lightgray!50} * \\
\hline
4 & \vee^{17} S^1 & \vee^{9}S^3 & \cellcolor{RoyalBlue!55}S^7 & \cellcolor{lightgray!50} * & \cellcolor{lightgray!50} * & \cellcolor{lightgray!50} * & \cellcolor{lightgray!50} * & \cellcolor{lightgray!50} * & \cellcolor{lightgray!50} * \\
\hline
5 & \vee^{26}S^1 & \cellcolor{teal!50}\bigvee^{9}S^2 & \vee^{9}S^4 & \cellcolor{lightgray!50} * & \cellcolor{lightgray!50} * & \cellcolor{lightgray!50} * & \cellcolor{lightgray!50} * & \cellcolor{lightgray!50} * & \cellcolor{lightgray!50} * \\
\hline
6 & \vee^{37}S^1 &\cellcolor{purple!60}\vee^{23}S^2 & 3\dv 1, 5\dv 12 & 5\dv 23, 8\dv 2 & \cellcolor{RoyalBlue!55}S^{17} & \cellcolor{lightgray!50} * & \cellcolor{lightgray!50} * & \cellcolor{lightgray!50} * & \cellcolor{lightgray!50} * \\
\hline
7 & \vee^{50}S^1 & \cellcolor{lightgray!50}\T^2 & 3\dv 1, 4\dv 14 & S^3 & 11\dv 1 & \cellcolor{lightgray!50} * & \cellcolor{lightgray!50} * & \cellcolor{lightgray!50} * & \cellcolor{lightgray!50} * \\
\hline
8 & \vee^{65}S^1 & \cellcolor{lightgray!50}\T^2 & \cellcolor{teal!50}\bigvee^{15}S^2\vee\bigvee^{16}S^3 & 3\dv 1, 7\dv 16\ (12\dv0) & 3\dv 1\ (12\dv0) & 12\dv 259 & \cellcolor{RoyalBlue!55}S^{31} & \cellcolor{lightgray!50} * & \cellcolor{lightgray!50} * \\
\hline
9 & \vee^{82}S^1 & \cellcolor{lightgray!50}\T^2 & \cellcolor{purple!60}\vee^{53}S^2 & 3\dv 1, 5\dv 36\ (9\dv0) & 3\dv 1\ (9\dv 0) & 7\dv 31, 8\dv 183\ (9\dv0) & (6\dv 0) & \cellcolor{lightgray!50} * & \cellcolor{lightgray!50} * \\
\hline
10 & \vee^{101}S^1 & \cellcolor{lightgray!50}\T^2 & \cellcolor{lightgray!50}\T^2 & 3\dv 21,
4\dv 60 & 3\dv 1, 9\dv 20 & 4\dv 39\ (10\dv0) & (6\dv 0)& (6\dv 0)& \cellcolor{RoyalBlue!55}S^{49} \\
\hline
11 & \vee^{122}S^1 & \cellcolor{lightgray!50}\T^2 & \cellcolor{lightgray!50}\T^2 & \cellcolor{teal!50}\bigvee^{21}S^2\vee\bigvee^{22}S^3 & 3\dv 1, 6\dv 22 & 3\dv 1\ (10\dv 0) & 4\dv1, 5\dv 4 & (5\dv 0) & (5\dv 0) \\
\hline
12 & \vee^{145}S^1 &\cellcolor{lightgray!50}\T^2 & \cellcolor{lightgray!50}\T^2 &\cellcolor{purple!60}\vee^{95}S^2 & 3\dv 1, 4\dv 24, 5\dv 120
& 3\dv 1\ (9\dv0) & 3\dv 1\ (6\dv0) & (4\dv 0) & (4\dv 0) \\
\hline
13 & \vee^{170}S^1 & \cellcolor{lightgray!50}\T^2 & \cellcolor{lightgray!50}\T^2 & \cellcolor{lightgray!50}\T^2 & 3\dv 27, 4\dv 26 & 3\dv 1, 8\dv26 & 3\dv 1\ (6\dv0) & 3\dv 1\ (4\dv0) & (4\dv 0) \\
\hline
14 & \vee^{197}S^1 &\cellcolor{lightgray!50}\T^2 & \cellcolor{lightgray!50}\T^2 & \cellcolor{lightgray!50}\T^2 & \cellcolor{teal!50}\bigvee^{27}S^2\vee\bigvee^{28}S^3& 3\dv 1, 5\dv28, 6\dv84 & 3\dv 1\ (6\dv0)& 3\dv 1\ (6\dv0) & 4\dv 1 \\
\hline
15 & \vee^{226}S^1 & \cellcolor{lightgray!50}\T^2 & \cellcolor{lightgray!50}\T^2 & \cellcolor{lightgray!50}\T^2 & \cellcolor{purple!60}\vee^{149}S^2& 3\dv 1, 4\dv 90& 3\dv 1\ (4\dv 0)& 3\dv 1\ (4\dv0) & 4\dv 89 \\
\hline
16 & \vee^{257}S^1 & \cellcolor{lightgray!50}\T^2 & \cellcolor{lightgray!50}\T^2 & \cellcolor{lightgray!50}\T^2 & \cellcolor{lightgray!50}\T^2 & 3\dv 35, 4\dv34\ (6\dv0)& 3\dv 1, 6\dv64\ (7\dv0)& 3\dv1\ (6\dv0) & 3\dv 1\ (4\dv 0) \\
\hline
17 & \vee^{290}S^1 & \cellcolor{lightgray!50}\T^2 & \cellcolor{lightgray!50}\T^2 & \cellcolor{lightgray!50}\T^2 & \cellcolor{lightgray!50}\T^2 &\cellcolor{teal!50}\bigvee^{33}S^2\vee\bigvee^{34}S^3 &3\dv 1, 4\dv34, 5\dv34 & 3\dv 1\ (6\dv0) & 3\dv 1\ (4\dv 0) \\
\hline
18 & \vee^{325}S^1 & \cellcolor{lightgray!50}\T^2 & \cellcolor{lightgray!50}\T^2 & \cellcolor{lightgray!50}\T^2 & \cellcolor{lightgray!50}\T^2 & \cellcolor{purple!60}\vee^{215}S^2 &3\dv3, 4\dv2, 5\dv216 & 3\dv 1\ (6\dv0)& 3\dv 1\ (4\dv 0) \\
\hline
19 & \vee^{362}S^1 & \cellcolor{lightgray!50}\T^2 & \cellcolor{lightgray!50}\T^2 & \cellcolor{lightgray!50}\T^2 & \cellcolor{lightgray!50}\T^2 & \cellcolor{lightgray!50}\T^2 &3\dv 39, 4\dv38\ (5\dv0) &3\dv 1\ (4\dv 0)& 3\dv 1\ (4\dv 0) \\
\hline
20 & \vee^{401}S^1 & \cellcolor{lightgray!50}\T^2 & \cellcolor{lightgray!50}\T^2 & \cellcolor{lightgray!50}\T^2 & \cellcolor{lightgray!50}\T^2 & \cellcolor{lightgray!50}\T^2 &\cellcolor{teal!50}\bigvee^{39}S^2\vee\bigvee^{40}S^3 & 3\dv 1, 4\dv160\ (5\dv0)& 3\dv 1\ (5\dv 0) \\
\hline
21 & \vee^{442}S^1 & \cellcolor{lightgray!50}\T^2 & \cellcolor{lightgray!50}\T^2 & \cellcolor{lightgray!50}\T^2 & \cellcolor{lightgray!50}\T^2 & \cellcolor{lightgray!50}\T^2 & \cellcolor{purple!60}\vee^{293}S^2 & 3\dv 3, 4\dv2, 5\dv294& 3\dv 1\ (5\dv0) \\
\hline
22 & \vee^{485}S^1 & \cellcolor{lightgray!50}\T^2 & \cellcolor{lightgray!50}\T^2 & \cellcolor{lightgray!50}\T^2 & \cellcolor{lightgray!50}\T^2 & \cellcolor{lightgray!50}\T^2 & \cellcolor{lightgray!50}\T^2 & 3\dv 47, 4\dv46\ (5\dv0)& 3\dv1, 4\dv44, 5\dv44\\
\hline
23 & \vee^{530}S^1 & \cellcolor{lightgray!50}\T^2 & \cellcolor{lightgray!50}\T^2 & \cellcolor{lightgray!50}\T^2 & \cellcolor{lightgray!50}\T^2 & \cellcolor{lightgray!50}\T^2 & \cellcolor{lightgray!50}\T^2 & \cellcolor{teal!50}\bigvee^{45}S^2\vee\bigvee^{46}S^3& 3\dv3, 4\dv2\ (5\dv0) \\
\hline
24 & \vee^{577}S^1 & \cellcolor{lightgray!50}\T^2 & \cellcolor{lightgray!50}\T^2 & \cellcolor{lightgray!50}\T^2 & \cellcolor{lightgray!50}\T^2 & \cellcolor{lightgray!50}\T^2 & \cellcolor{lightgray!50}\T^2 & \cellcolor{purple!60}\vee^{383}S^2& 3\dv 3, 4\dv2, 5\dv384 \\
\hline
25 & \vee^{626}S^1 & \cellcolor{lightgray!50}\T^2 & \cellcolor{lightgray!50}\T^2 & \cellcolor{lightgray!50}\T^2 & \cellcolor{lightgray!50}\T^2 & \cellcolor{lightgray!50}\T^2 & \cellcolor{lightgray!50}\T^2 & \cellcolor{lightgray!50}\T^2 & 3\dv 51, 4\dv50\ (5\dv0) \\
\hline
26 & \vee^{677}S^1 & \cellcolor{lightgray!50}\T^2 & \cellcolor{lightgray!50}\T^2 & \cellcolor{lightgray!50}\T^2 & \cellcolor{lightgray!50}\T^2 & \cellcolor{lightgray!50}\T^2 & \cellcolor{lightgray!50}\T^2 & \cellcolor{lightgray!50}\T^2 & \cellcolor{teal!50}\bigvee^{51}S^2\vee\bigvee^{52}S^3 \\
\hline
27 & \vee^{730}S^1 & \cellcolor{lightgray!50}\T^2 & \cellcolor{lightgray!50}\T^2 & \cellcolor{lightgray!50}\T^2 & \cellcolor{lightgray!50}\T^2 & \cellcolor{lightgray!50}\T^2 & \cellcolor{lightgray!50}\T^2 & \cellcolor{lightgray!50}\T^2 & \cellcolor{purple!60}\vee^{485}S^2 \\
\hline
28 & \vee^{785}S^1 & \cellcolor{lightgray!50}\T^2 & \cellcolor{lightgray!50}\T^2 & \cellcolor{lightgray!50}\T^2 & \cellcolor{lightgray!50}\T^2 & \cellcolor{lightgray!50}\T^2 & \cellcolor{lightgray!50}\T^2 & \cellcolor{lightgray!50}\T^2 & \cellcolor{lightgray!50}\T^2 \\
\Xhline{1.7pt}
\end{tabular}}
\end{adjustwidth}
\end{center}
\captionsetup{width=1\linewidth}
\caption{Homotopy types of clique complexes of graph powers of the torus grid graphs $\vr{T_{n,n}}{k}$.
We write
$i\dv d$ as shorthand for $H_i(\vr{T_{n,n}}{k};\Z/2) \cong (\Z/2)^d$.
Cells that contain an entry of the form $(i\dv0)$ have been computed up to the $i$-th homology group without yielding additional features.
Cells not containing $(i\dv0)$ have had the homology computed up to the greatest $i$ within the cell.
All computations were done using $\Z/2$ coefficients in Ripser.
Each distinct colored diagonal corresponds to a result discussed in the paper: the \colorbox{RoyalBlue!50}{blue}-highlighted diagonal is explained in Section~\ref{sec:intro}, while the \colorbox{purple!60}{red} and \colorbox{teal!50}{teal} diagonals are explained in Section~\ref{sec:wedgeS2S3}.}
\label{table:Torus-Grids}
\end{table}

\begin{table}[h]
\scalebox{0.8}{
\begin{tabular}{? >{$} c <{$} ? >{$} c <{$} | >{$} c <{$} | >{$} c <{$} | >{$} c <{$} ?}
%{| >{} c <{} | >{} c <{} | >{} c <{} | >{} c <{} | >{} c <{} |}
\Xhline{1.7pt}
%\cl(T_{n,n}^k) 
& k=10 & 11 & 12 & 13 \\
\Xhline{1.7pt}
n=12 & (4\dv 0) & \cellcolor{RoyalBlue!55}S^{71} & \cellcolor{lightgray!50} * & \cellcolor{lightgray!50} *\\
\hline
13 & (4\dv 0)& (4\dv 0) & \cellcolor{lightgray!50} * & \cellcolor{lightgray!50} *\\
\hline
14 & (4\dv 0)& (4\dv 0) & (4\dv 0)& \cellcolor{RoyalBlue!55}S^{97}\\
\hline
15 & (4\dv 0) & (4\dv 0) & (4\dv 0) & (4\dv 0)\\
\hline
16 & 4\dv 3 & (4\dv 0) & (4\dv 0) & (4\dv 0)\\
\hline
17 & 3\dv 1\ (4\dv 0) & 4\dv 1 & (4\dv 0) & (4\dv 0) \\
\hline
18 & 3\dv 1\ (4\dv 0) & 3\dv 1\ (4\dv 0) & (4\dv 0) & (4\dv 0)\\
\hline
19 & 3\dv 1\ (4\dv 0) & 3\dv 1\ (4\dv 0) & 4\dv 1&\\
\hline
20 & 3\dv 1\ (4\dv 0) & 3\dv 1\ (4\dv 0) & 4\dv 159 &\\
\hline
21 & 3\dv 1\ (4\dv 0) & 3\dv 1\ (4\dv 0) & 3\dv 1\ (4\dv 0) &\\
\hline
22 & 3\dv 1\ (4\dv 0) & 3\dv 1\ (4\dv 0) & 3\dv 1\ (4\dv 0) &\\
\hline
23 & 3\dv 1\ (4\dv 0) & 3\dv 1\ (4\dv 0) & 3\dv 1\ (4\dv 0) &\\
\hline
24 & 3\dv 1\ (4\dv 0) & 3\dv 1\ (4\dv 0) &&\\
\hline
25 & 3\dv 1, 4\dv 250 & 3\dv 1\ (4\dv 0) &&\\
\hline
26 & 3\dv 3, 4\dv 2& 3\dv 1\ (4\dv 0) &&\\
\hline
27 & 3\dv 3, 4\dv 2 & 3\dv 1, 4\dv 54 &&\\
\hline
28 & 3\dv 59, 4\dv 58& 3\dv 3, 4\dv 2 &&\\
\Xhline{1.7pt}
\end{tabular}
}
\caption{A continuation of Table~\ref{table:Torus-Grids}.
Additional computational power is likely needed to find interesting topology in the blank entries.
}
\label{table:cont}
\end{table}

Let us describe some initial examples in Table~\ref{table:Torus-Grids}.
At scale $0$, we have that $\vr{T_{n,n}}{0}$ is a collection of $n^2$ vertices, and therefore $\vr{T_{n,n}}{0} \cong \vee^{n^{2}-1}S^{0}$.

By an Euler characteristic argument, any connected graph with $v$ vertices and $e$ edges is homotopy equivalent to $\vee^{e-v+1} S^1$.
Since $\vr{T_{n,n}}{1}$ has $n^2$ vertices and $2n^2$ edges, we get that 
\[
\vr{T_{n,n}}{1}\simeq \vee^{e-v+1} S^1\simeq \vee^{n^{2}+1}S^{1}\quad \text{for} \quad n\geq 4.
\]

For $k\ge 2$ and for $n>3k$, i.e.\ when the scale parameter $k$ is small enough compared to $n$, we prove in Theorem~\ref{thm:torus_homotopy} that $\vr{T_{n,n}}{k}$ is homotopy equivalent to the torus $\T^2$.

We note that $\vr{T_{n,n}}{k}$ is contractible if the scale parameter $k$ is large enough:
\[
\vr{T_{n,n}}{k}\simeq\ast \quad\text{for }k\ge
\begin{cases}
n &\text{if }n\text{ is even} \\
n-1 &\text{if }n\text{ is odd.}
\end{cases}
\]
Indeed, this follows since the diameter $\diam(T_{n,n})$ is equal to $n$ if $n$ is even, and equal to $n-1$ if $n$ is odd.
Once $k$ is large enough so that $k\ge \diam(T_{n,n})$, then any two vertices in $T_{n,n}$ are connected by an edge in the Vietoris--Rips complex, meaning that $\vr{T_{n,n}}{k}$ is a complete simplex and hence contractible.

Interesting topology arises over intermediate scales $\frac{n-1}{3}<k<2\lfloor\frac{n}{2}\rfloor$.
Here $2\lfloor\frac{n}{2}\rfloor$ is simply the function returning $n$ if $n$ is even, and returning $n-1$ if $n$ is odd --- i.e., the function returning the smallest even integer less than or equal to $n$.
Two goals of our paper are (i) to describe the topology of $\vr{T_{n,n}}{k}$ for $\frac{n-1}{3}<k<2\lfloor\frac{n}{2}\rfloor$ as best we can, and (ii) to ask open questions about aspects of the topology in this regime that we have observed (say from homology computations).
We hope this paper inspires future work studying the intricate topology that arises.

For $n$ even we have a homeomorphism $\vr{T_{n,n}}{n-1}\cong S^{\frac{n^2}{2}-1}$; the entries in Table~\ref{table:Torus-Grids} that are examples of this are $\vr{T_{4,4}}{3}\cong S^{7}$, $\vr{T_{6,6}}{5}\cong S^{17}$, $\vr{T_{8,8}}{7}\cong S^{31}$, and $\vr{T_{10,10}}{9}\cong S^{49}$.
These homeomorphisms follow since the complex in question is isomorphic to the boundary of a cross-polytope.
Indeed, if $n$ is even, then $\diam(T_{n,n})=n$, and so at scale $\diam(T_{n,n})-1 = n-1$, each vertex in $\vr{T_{n,n}}{n-1}$ is connected to every other vertex besides its ``antipodal pair''.
By Corollary~\ref{cor:cross-polytopal}, this means that $\vr{T_{n,n}}{k}$ is isomorphic to the boundary of a cross-polytope on $n^2$ vertices, hence homeomorphic to the boundary of a ball of dimension $\frac{n^2}{2}$, and hence homeomorphic to a sphere of dimension $\frac{n^2}{2}-1$.

We prove that $\vr{T_{3n,3n}}{n}\simeq\vee^{6n^2-1}S^2$ for $n\ge 2$.
These homotopy types arise as follows.
The simplicial complex $\vr{T_{3n,3n}}{n}$ is obtained from a subcomplex that is homotopy equivalent to the torus (which we prove using the nerve lemma) by attaching $6n^2$ additional $2$-simplices.
Of these $6n^2$ additional disks, $3n^2$ are attached with boundary circle wrapping once around one circular factor of the torus $S^1\times S^1$, and the remaining $3n^2$ are attached with boundary circle wrapping once around the other circular factor of the torus.
This produces a space that is homotopy equivalent to the $(6n^2-1)$-fold wedge sum of $2$-spheres.

Using a similar proof technique, we also prove that $\vr{T_{3n-1,3n-1}}{n}\simeq \bigvee_{6n-3} S^2\vee \bigvee_{6n-2}S^3$ for $n\geq 3$.

Based on homology computations, we conjecture that for $n\ge 3$ the reduced homology of $\vr{T_{3n-2,3n-2}}{n}$ is nontrivial only in dimensions $3$ and $4$.

In Section~\ref{sec:homotopy-through-homology}, we use homology computations with integer coefficients and the theorems of Hurewicz and Whitehead to prove the homotopy equivalences $\vr{T_{5,5}}{3}\simeq \bigvee_9 S^4$ and $\vr{T_{7,7}}{4}\simeq S^3$.
We conjecture that $\vr{T_{n,n}}{k}$ is homotopy equivalent to a $3$-sphere for a countable family of $(n,k)$ pairs; see Question~\ref{ques:3sphere}.
We conjecture that this 3-sphere is formed from the hollow torus $S^1 \times S^1$ by gluing in two solid tori (the space fills in as the scale increases), in order to obtain the $3$-sphere as the standard genus-1 Heegaard decomposition $S^3=(S^1\times D^2)\cup_{S^1\times S^1}(D^2 \times S^1)=S^1 * S^1$.

Even though Table~\ref{table:Torus-Grids} and~\cite[Table~1]{adamaszek2022vietoris} are quite different, one may recognize a similarity between the homotopy types in one row and one column selected from each table:
\begin{center}
\begin{tabular}{|rl|rl|}
\hline
$4$-th \ row&of \ Table~\ref{table:Torus-Grids} &
$4$-th \ column&of \ \cite[Table~1]{adamaszek2022vietoris} \\
\hline
$\vr{T_{4,4}}{0}$&$\simeq\ \ \vee^{15} S^0$ &
$\vr{Q_4}{0}$&$\simeq\ \ \vee^{15} S^0$ \\
$\vr{T_{4,4}}{1}$&$\simeq\ \ \vee^{17} S^1$ &
$\vr{Q_4}{1}$&$\simeq\ \ \vee^{17} S^1$ \\
$\vr{T_{4,4}}{2}$&$\simeq\ \ \vee^9 S^3$ &
$\vr{Q_4}{2}$&$\simeq\ \ \vee^9 S^3$ \\
$\vr{T_{4,4}}{3}$&$\simeq\ \ S^7$ &
$\vr{Q_4}{3}$&$\simeq\ \ S^7$ \\
$\vr{T_{4,4}}{k}$&$\simeq\ \ *$ for $k\ge 4$ &
$\vr{Q_4}{k}$&$\simeq\ \ *$ for $k\ge 4$ \\
\hline
\end{tabular}
\end{center}
Here $Q_4$ is the vertex set of the $4$-dimensional hypercube with $2^4=16$ vertices, equipped with the Hamming metric, and the homotopy types of $\vr{Q_4}{k}$ are proven for all $k$ in~\cite{adamaszek2022vietoris}.
This similarity is no coincidence.
Indeed, we observe that the $4\times 4$ torus grid graph $T_{4,4}$, which also has 16 vertices, is isomorphic to the $4$-dimensional hypercube graph $Q_4$.
It follows immediately that $\vr{T_{4,4}}{k}\cong \vr{Q_4}{k}$ for all $k$, and in particular, $\vr{T_{4,4}}{2}\simeq\vee^9 S^3$ and $\vr{T_{4,4}}{3}\simeq S^7$.

More generally, let $Q_n$ be the vertex set of the hypercube graph with $2^n$ vertices equipped with the shortest distance path metric, or equivalently, the set of all $2^{n}$ binary strings of length $n$ equipped with the Hamming distance.
The $\underbrace{4\times\ldots\times 4}_{m\text{ times}}$ grid graph on the $m$-dimensional torus $\T^m=(S^1)^m$ is isomorphic to the $2m$-dimensional hypercube graph $Q_{2m}$, and so their Vietoris--Rips complexes coincide as well.
So, determining the homotopy types of Vietoris--Rips complexes of grid graphs on the $m$-dimensional torus can be thought of as a generalization of the problem of determining the homotopy types of Vietoris--Rips complexes of (even-dimensional) hypercube graphs, which is a hard problem~\cite{adamaszek2022vietoris,shukla2022vietoris,adams2024lower,feng2023homotopy,feng2023vietoris,feng2024exploring}.

We remark that results about the equivariant homotopy type of Vietoris--Rips and \v{C}ech complexes of tori may have consequences for the roots of multivariate trigonometric polynomials (see~\cite{cantor2025roots}).
We also refer the reader to~\cite{fendley2004hard,jonsson2006hard,IndCplxSquareGrids} for results on independence complexes of square grid graphs.

We begin in Section~\ref{sec:related} with a description of related work, before continuing with preliminaries and notation in Section~\ref{sec:preliminaries}.
Section~\ref{sec:Z^2} describes the maximal simplices in $\vr{\Z^2}{2}$.
In Section~\ref{sec:wedgeS2S3} we prove the homotopy types of
the torus family $\vr{T_{n, n}}{k}\simeq\T^2$ for $k\ge 2$ and $n>3k$, the family $\vr{T_{3n,3n}}{n}\simeq\vee^{6n^2-1}S^2$ for $n\ge 2$, and the family $\vr{T_{3n-1,3n-1}}{n}\simeq \bigvee_{6n-3} S^2\vee \bigvee_{6n-2}S^3$ for $n\geq 3$.
Section~\ref{sec:cross-polytopes} describes the connection to cross-polytopes, and Section~\ref{sec:homotopy-through-homology} explains how certain homotopy types can be deduced from integral homology computations.
We conclude in Section~\ref{sec:conclusion} with a list of open questions.

\section{Related work}
\label{sec:related}

\subsection*{Hausmann's theorem}

Hausmann's theorem~\cite{Hausmann1995} states that if $M$ is a compact Riemannian manifold, then there is a positive scale $r(M) > 0$ such that if $0 < r < r(M)$, then the Vietoris--Rips complex $\vr{M}{r}$ is homotopy equivalent to $M$.
The value $r(M)$ depends on the geometry of the manifold $M$, such as its curvature.
This is extended by Latschev's theorem~\cite{Latschev2001}, which further implies that if $X$ is sufficiently close to $M$ in the Gromov--Hausdorff distance, then the homotopy equivalence $\vr{X}{r}\simeq M$ is maintained when scale $r$ is sufficiently small compared to the curvature of $M$ and sufficiently large compared to the density of $X$.

\subsection*{Vietoris--Rips complexes of products}
Let $(X_{1},\rho_{1}), (X_{2},\rho_{2}),\ldots,(X_{n},\rho_{n})$ be metric spaces.
For each $1\le p \leq\infty$ we define the $l^p$ metric on the product $X=X_{1}\times X_{2}\times \ldots\times X_{n}$ by
\[d_{l^p}((x_{1},x_{2},\ldots,x_{n}),(y_{1},y_{2},\ldots,y_{n}))=\bigg(\sum_{i=1}^{n}(\rho_{i}(x_{i},y_{i}))^{p}\bigg)^{1/p}.\]
In the case $p=\infty$, we have
\[d_{l^\infty}((x_{1},x_{2},\ldots,x_{n}),(y_{1},y_{2},\ldots,y_{n}))=\max\{\rho_{i}(x_{i},y_{i})~|~1\le i\le n\}.\]
By~\cite[Proposition~10.2]{AA-VRS1}, if we equip the product with the $l^\infty$ metric, then for any $r>0$, we have a homotopy equivalence 
\[\vr{X}{r}\simeq\vr{X_{1}}{r}\times\ldots\times\vr{X_{n}}{r}.\]

Little is known if we instead equip the product with the $l^p$ metric for $p<\infty$.
For example, if $I=\{0,1\}$ is a two-point metric space with $d_{I}(0,1)=1$, then the Vietoris--Rips complexes are easy to understand: 
$\vr{I}{r}$ is two points (the zero-dimensional sphere) for $r<1$, 
and $\vr{I}{r}$ is a (contractible) edge for $r\ge 1$.
So, by~\cite[Proposition~10.2]{AA-VRS1}, $\vr{(I^n,l^\infty)}{r}$ is $2^n$ points for $r<1$, 
and $\vr{(I^n,l^\infty)}{r}$ is contractible for $r\ge 1$.
But, if instead $p=1<\infty$, then the product $(I^n,l^1)$ is isometric to $Q_n$, the vertex set of the hypercube graph equipped with the shortest path metric.
And, the Vietoris--Rips complexes of these metric spaces $\vr{(I^n,l^1)}{r}=\vr{Q_n}{r}$ are quite complicated; see~\cite{adamaszek2022vietoris,shukla2022vietoris,adams2024lower,feng2023homotopy,feng2023vietoris,feng2024exploring}.

There are several papers exploring K\"unneth formulae for persistent homology~\cite{polterovich2017persistence,gakhar2019kunneth,carlsson2020persistent,bubenik2021homological,VargasRosarioThesis}, though from this perspective the Vietoris--Rips complexes of $l^\infty$ products are again significantly simpler than Vietoris--Rips complexes of $l^1$ products.

\subsection*{Vietoris--Rips complexes of platonic solids}

In~\cite{saleh2023vietoris}, the homotopy types of Vietoris--Rips complexes of platonic solids are determined.
In more detail, the vertex set of the platonic solid is equipped with the metric that arises from the shortest path metric in the graph that is the 1-skeleton of that platonic solid.
Many of the higher-dimensional homotopy types that arise are boundaries of cross-polytopes, at a scale parameter one below the diameter of the vertex set.
The most new interesting shape that arises is the Vietoris--Rips complex of the vertex set of the dodecahedron, at scale parameter two below the diameter, when it is homotopy equivalent to a 9-fold wedge of $3$-spheres.

\subsection*{Vietoris-Rips complexes of grid graphs}

Let $\Z^n$ be the integer lattice in $n$-dimensional space, equipped with the $l^1$ metric.
Zaremsky asked the question if for any $n$, there exists a scale parameter $r$ sufficiently large such that the Vietoris--Rips complex $\vr{\Z^n}{r}$ is contractible.
This result was recently proven by Virk in~\cite{virk2024contractibility}; see also~\cite{zaremsky2024contractible}.
The case of $n=2$ or $n=3$ is simpler; see~\cite{373007,wang2024contractibilityvietorisripscomplexesdense}.
In order to understand lattice grids in a torus, we need a different proof in the $n=2$ case that provides a characterization of the maximal simplices of $\vr{\Z^2}{r}$, which we do in Section~\ref{sec:Z^2}.

\section{Preliminaries}
\label{sec:preliminaries}

In this section, we introduce the definitions and notation for metric spaces, simplicial complexes, graph theory, and topology that we will need in the remainder of the paper.

\subsection{Metric spaces}
\label{ssec:metric-spaces}

We denote the \emph{open ball of radius $r\ge 0$ centered at $x\in X$} by $B_X(x,r) = \{ y\in X : d(x,y) < r\}$, and the \emph{closed ball} by $B_X[x,r] = \{ y\in X : d(x,y) \le r\}$.
When the space $X$ is clear, we may write these open and closed balls as $B(x,r)$ and $B[x,r]$.
We define the \emph{diameter} of subset $A\subseteq X$ as $\diam(A)=\sup_{x, y \in A }d(x,y)$.
If $\diam(A)$ is not finite, then we set $\diam(A) = \infty$.

The most frequent metric spaces we will encounter in this paper are the following.
We equip the integer lattice $\Z^2$ with the $l^1$ metric, defined by $d((x,y),(x',y'))=|x-x'|+|y-y'|$ for all $(x,y),(x',y')\in \Z^2$.
This is also referred to as the word metric on $\Z^2$.
We also equip $\R^2$ with the $l^1$ metric, defined in the same way for all $(x,y),(x',y')\in \R^2$.
We let $C_n$ denote the set $\{0,1,\ldots,n-1\}$, equipped with the circular metric $d(i,i')=\min\{|i-i'|,n-|i-i'|\}$.
We let $T_{n,n}$ denote the set $\{0,1,\ldots,n-1\} \times \{0,1,\ldots,n-1\}$, equipped with the $l^1$ product metric, namely $d((i,j),(i',j'))=\min\{|i-i'|,n-|i-i'|\}+\min\{|j-j'|,n-|j-j'|\}$.
The diameter of $C_n$ is $\frac{n}{2}$ if $n$ is even and $\frac{n-1}{2}$ if $n$ is odd.
It follows that the diameter $\diam(T_{n,n})$ is equal to $n$ if $n$ is even, and equal to $n-1$ if $n$ is odd.
Note that the metric on $C_n$ recovers the shortest path metric on the vertex set of a cycle graph, and that the metric on $T_{n,n}$ recovers the shortest path metric on the vertex set of a torus grid graph.
For this reason, by an abuse of notation, we will also use the symbols $C_n$ and $T_{n,n}$ to denote graphs in Section~\ref{ssec:graphs}.

\subsection{Simplicial complexes}
A \emph{$k$-simplex} is a $k$-dimensional polytope that is the convex hull of its $k + 1$ vertices.
For example, a 0-dimensional simplex is a point, a 1-dimensional simplex is a line segment, a 2-dimensional simplex is a triangle, a 3-dimensional simplex is a tetrahedron, etc.

The convex hull of any nonempty subset of the $k + 1$ points that define an $k$-simplex is called a \emph{face} of the simplex.
A \emph{simplicial complex} $K$ is a set of simplices of an underlying vertex set that satisfies the following conditions: 
(1) every face of a simplex from $K$ is also in $K$, and 
(2) the nonempty intersection of any two simplices in $K$ is a face of both of them.
A simplex $\sigma\in K$ is a \emph{maximal simplex} or a \emph{facet} in $K$ if there is no simplex $\tau\in K$ with $\sigma\subsetneq\tau$.
That is, a facet is a simplex that is not the face of any larger simplex.

An abstract simplicial complex $K$ consists of a vertex set $V$ and a collection of simplices, namely finite subsets of $V$, such that if $\sigma\in K$ and $\sigma'\subseteq \sigma$, then $\sigma'\in K$.
We identify abstract simplicial complexes with their geometric realizations, described in the paragraph above.

\subsection{Vietoris--Rips complexes}

For a metric space $(X,d)$ and $r\geq 0$, define the \emph{Vietoris-Rips complex}, denoted as $\vr{X}{r}$, as the simplicial complex satisfying:
\begin{itemize}
\item the underlying vertex set of $\vr{X}{r}$ is $X$, and
\item a finite set $\sigma\subseteq X$ is a simplex in $\vr{X}{r}$ if and only if $\diam(\sigma) \leq r$.
\end{itemize}
This is the closed Vietoris-Rips complex; the open Vietoris-Rips complex instead uses $\diam(\sigma)<r$.

\subsection{Nerve complexes}

We will sometimes understand the homotopy type of a space using nerve complexes and the nerve lemma, which we introduce next.

Given a finite collection of sets $\cU=\{U_{\alpha}\}_{\alpha\in A}$ with each $U_\alpha\subseteq X$,
the \emph{nerve complex} of $\cU$ is the simplicial complex $N(\cU)$ 
whose vertex set is the index set $A$, and where a subset $\{\alpha_{0},\alpha_{1},...,\alpha_{k}\}\subseteq A$ spans a $k$-simplex in $N(\cU)$  if and only if $U_{\alpha_{0} }\cap U_{\alpha_{1} }\cap ...\cap U_{\alpha_{k} }\neq\emptyset$.

\begin{theorem}[Nerve Theorem]
\label{thm:nerve}
Let $\cU$ be a collection of open sets in a metric space $X$ with $\cup \cU = X$.
If every nonempty intersection $\bigcap^{n}_{i=0} U_{\alpha_{i}}$ is contractible, then the nerve simplicial complex $N(\cU)$ is homotopy equivalent to $X$.
\end{theorem}

This version of the nerve theorem follows from~\cite[Corollary~4G.3]{Hatcher} since every metric space is paracompact.
We refer the reader to~\cite{Borsuk1948,Bjorner1995} for other versions of the nerve theorem.

\subsection{Graphs}
\label{ssec:graphs}

A \emph{graph} is a pair $G = (V,E)$ consisting of a set of
vertices, $V$, and a set of edges, $E$, where each edge is a pair of vertices.
In some cases, we will use the notation $G = (V(G), E(G))$ to emphasize that the vertices and edges correspond to the graph $G$.
The \emph{shortest path metric} $d_G$ on $V(G)$ is the metric satisfying that $d_G(u,v)$ is the minimum length of a path of edges connecting $u$ to $v$.
If $u$ and $v$ are not connected by a path of edges, we say $d_G(u,v) = \infty$.
The \emph{diameter} of a finite connected graph $G$ is defined as $\diam(G) = \max \{ d_G(u,v) : u,v \in G\}$.

Given a graph $G$ and a positive integer $k$, the \emph{graph power} $G^k$ is the graph with $V(G^k) = V(G)$, and with $(u,v) \in E(G^k)$ if and only if $d_G(u,v) \leq k$.

For an integer $n\ge 3$, we define the \emph{cycle graph} $C_n = (V(C_n), E(C_n))$, where $V(C_n) = \{ 0, 1, \ldots, n-1 \}$ and $E(C_n) = \{ (i, i+1) : 0 \le i \le n-1 \} \cup \{ (n-1, 0)\}$.

Given graphs $G= (V(G), E(G))$ and $H=(V(H), E(H))$, define the \emph{Cartesian product} (sometimes called \emph{box product}) $G\ \square \ H$ as the graph such that: (i) $V(G \ \square \ H) = V(G) \times V(H)$, and (ii) $((u,u' ),(v,v')) \in E(G\ \square \ H)$ if and only if either $u = v$ and $(u',v') \in E(H)$, or $u' = v'$ and $(u,v)\in E(G)$.
The \emph{torus grid graph} of $n$ by $m$ vertices is defined as $T_{n,m} = C_n \ \square \ C_m$.

The \emph{clique complex} $\cl(G)$ of a graph $G$ is the simplicial complex with
\begin{itemize}
\item vertex set $V(G)$, and
\item a finite $\sigma$ as a simplex if and only if $\sigma$ induces a complete subgraph of $G$.
\end{itemize}

We can relate the clique complexes of graph powers to Vietoris-Rips complexes of graphs equipped with the shortest path metric.
Indeed, for every $k\ge 0$, the clique complex $\cl(G^k)$ is equal to the Vietoris--Rips complex $\vr{(V(G), d_G)}{k}$.
For ease of notation, we write $\vr{(V(G), d_G)}{k}$ as $\vr{G}{k}$.
In particular, $\vr{T_{n,n}}{k}$ denotes the same simplicial complex whether $T_{n,n}$ is interpreted as the metric space defined in Section~\ref{ssec:metric-spaces} or as the vertex set of the torus grid graph equipped with the shortest path metric.

%Notice that the diameter of the cyclic graph, $C_n$, is $\diam(C_n)=\frac{n}{2}$ if $n$ is even, and $\diam(C_n) =\frac{n-1}{2}$ if $n$ is odd.

%\begin{lemma}
%\label{Lemma:DiameterCycleGraph}
%Let $n, m \geq 1$.
%The diameter of a product of cyclic graphs $C_n \ \square \ C_m$ is the sum of the diameters.
%That is, $\diam(C_n \ \square \ C_m) = \diam(C_n) +\diam(C_m)$.
%\end{lemma}

%\begin{proof}
%Given any two vertices $(u, u'), (v, v')$ in $C_n \ \square \ C_m$, there is a path of length $\leq \diam(C_n)$ connecting $u$ with $v$ and a path of length $\leq \diam(C_m)$ connecting $u'$ with $v'$.
%Hence, there is a path of length $\leq \diam(C_n) +\diam(C_m)$.
%The equality in the conclusion is satisfied when $d(u,v) = \diam(C_n)$ and $d(u',v') = \diam(C_m)$.
%\end{proof}

\subsection{Homology}

Given a simplicial complex $K$, we define $\Delta_n(K)$ to be the free abelian group consisting of finite integral combinations of the oriented $n$-simplices in $K$.
The \emph{boundary homomorphism} $\partial_n:\Delta_n(K)\rightarrow\Delta_{n-1}(K)$ is a function that maps each oriented $n$-simplex to a linear combination of its $(n-1)$-dimensional faces, each with an appropriate sign to respect orientation, such that $\partial_{n-1}\circ\partial_{n}=0$.
The collection of all $\Delta_i(K)$ connected by boundary homomorphisms, $\partial_i$, form a \emph{chain complex} 
 \[\cdots\xrightarrow{\partial_{n+2}}\Delta_{n+1}(K)\xrightarrow{\partial_{n+1}}\Delta_n(K)\xrightarrow{\partial_n} \Delta_{n-1}(K)\xrightarrow{\partial_{n-1}}\cdots\xrightarrow{\partial_2}\Delta_1(K)\xrightarrow{\partial_1} \Delta_0(K)\rightarrow 0\]
 where $\partial_i\circ\partial_{i+1}=0$ for each $i$.
 
The \emph{$i$-th homology group} (with $\Z$ coefficients), $H_i(K)$, is defined to be the quotient group $\mathrm{Ker}\partial_i/\mathrm{Im}\partial_{i+1}$.
Moreover, the \emph{$i$-th Betti number} $\beta_i(K)$ is the rank of $H_i(K)$, i.e.\ $\beta_i(K)=\mathrm{rank}(H_i(K))$.

\subsection{Connectivity}

We say that a nonempty topological space $Y$ is \emph{$k$-connected} if the homotopy groups $\pi_i(Y)$ are trivial for all $i\le k$.

\section{The facets of $\vr{\Z^2}{k}$}
\label{sec:Z^2}

Recall that we equip $\Z^2$ and $\R^2$ with the $l^1$ metric, given by
\[d((x,y),(x',y'))=|x-x'|+|y-y'|.\]
In this section, we classify the facets of $\vr{\Z^2}{k}$ for any scale $k$, where $\Z^2$ is equipped with the $l^1$ metric.
This classification will be needed for the proofs of Theorems~\ref{thm:homotopy_3k} and~\ref{thm:homotopy_3k-1} in the following section.
As an easy consequence, we also recover the known result that $\vr{\Z^2}{k}$ is contractible for scales $k$ greater than or equal to $2$.
We begin by introducing the required notation and key lemmas.

Since $\vr{\Z^2}{k}=\vr{\Z^2}{\lfloor k\rfloor}$, it suffices to restrict attention to integer values of $k$.
Given a positive integer $k$, we say that a simplex $\sigma \subseteq \Z^2$ is \emph{$k$-generated} by a set of vertices $\hat \sigma \subseteq \sigma$ if $\sigma = \bigcap_{v\in \hat \sigma} B_{\Z^2}[v,k]$; call the set $\hat \sigma$ a \emph{generating set}.

In the set $\Z^2$, for $n, m \in \N$, let $R_{n,m} = \{ (i,j) \in \Z^2 : 0 \le i\le n-1, 0\le j\le m-1\}$ be the rectangle with $n$ vertices on the base and $m$ vertices on the height.
Note that $R_{1,1}$ is simply a point, $R_{2,1}$ is two vertices with distance $1$, and $R_{2,2}$ is the set of corners of a square.
We define the following collections of simplices with vertex sets in $\Z^2$ up to isometry (translation and $90^o$ rotation).

If $k$ is even, let 
\begin{itemize}
\item
$\Z ( \boldsymbol{\cdot} , k) = \{ \sigma \subseteq \Z^2 : \sigma \text{ is a $\frac{k}{2}$-generated simplex and } \hat \sigma \cong R_{1,1} \}$.
\item
$\Z (\square, k) = \{ \sigma \subseteq \Z^2 : \sigma \text{ is a $\frac{k+2}{2}$-generated simplex and } \hat \sigma \cong R_{2,2} \}$, and
\end{itemize}
If $k$ is odd, let 
\begin{itemize}
\item
$\Z ( \edge, k) = \{ \sigma \subseteq \Z^2 : \sigma \text{ is a $\frac{k+1}{2}$-generated simplex and } \hat \sigma \cong R_{2,1} \}$.
\end{itemize}
By $\hat \sigma \cong R_{n,m}$ we mean that $\hat \sigma$ is isometric to $R_{n,m}$.
Observe that $\Z ( \boldsymbol{\cdot} , k)$ is simply the collection of all closed balls of radius $\frac{k}{2}$ centered at a vertex in $\Z^2$.

Given a non-vertical line $L$ in the plane $\R^2$ of the form $y=mx+b$, define 
\[T_L = \{ (x,y) \in \Z^2 : y> mx+b \} \text{ and } 
B_L = \{ (x,y) \in \Z^2 : y< mx+b \}.\]
That is, $T_L$ and $B_L$ are the top and bottom parts of $\Z^2$ split by the line $L$.
We say that $L$ is \emph{tangent} to a set of vertices $\sigma \subseteq \R^2$ if (1) $L$ contains at least one point of $\sigma$, and (2) either
$(\sigma \subseteq T_L \cup L \text{ and } \sigma \cap B_L =\emptyset) \text{ or } (\sigma \subseteq B_L \cup L \text{ and } \sigma \cap T_L =\emptyset).$

\begin{lemma}
\label{Lemma:EnclosedSigma}
Let $k\ge 0$ be an integer and let $\sigma \in \vr{\Z^2}{k}$ be a maximal simplex.
Then there is a $c \in \R^2$ with each coordinate an integer or a half-integer so that $\sigma = B_{\R^2}[c,\frac{k}{2}] \cap \Z^2$.
\end{lemma}

\begin{proof}
Let $\sigma$ be a facet in $\vr{\Z^2}{k}$.
Now, let $L_1, L_2$ be lines in $\R^2$ of the form $y=x+k_1$, $y=x+k_2$ so that $L_1$ and $L_2$ are tangent to $\sigma$ and $\sigma \subseteq (B_{L_1} \cup L_1) \cap (T_{L_2} \cup L_2)$.
Note the lines $L_1$ and $L_2$ are parallel translations of the positive diagonal enclosing $\sigma$ in a sandwich fashion.
Also, let $L_3, L_4$ be lines in $\R^2$ of the form $y=-x+k_3$, $y=-x+k_4$ so that $L_3$ and $L_4$ are tangent to $\sigma$ and $\sigma \subseteq (B_{L_3} \cup L_3) \cap (T_{L_4} \cup L_4)$.
Similarly, the lines $L_3$ and $L_4$ are parallel translations of the negative diagonal enclosing $\sigma$.
Note $k_1,k_2,k_3,k_4\in\Z$.

Let $P$ denote the closed quadrilateral of points in $\R^2$ enclosed by the lines $L_i$ for $1\leq i \leq 4$.
This implies $\sigma \subseteq P$.

We claim that $d(L_1, L_2) = d(L_3, L_4) = k$.
On the one hand, if $d(L_1, L_2)> k$, then $\diam(\sigma) >k$: pick points $v_1 \in L_1 \cap \sigma$ and $v_2 \in L_2 \cap \sigma$, which exist since the lines are tangent to $\sigma$.
Because $L_1$ and $L_2$ are parallel, any pair $u \in L_1$ and $w\in L_2$ have distance $>k$ by our assumption, so $d(v_1, v_2) >k$ which is a contradiction since $v_1, v_2 \in \sigma$.
Similarly for $L_3$ and $L_4$.
So, we can assume that both $d(L_1, L_2), d(L_3, L_4) \leq k$.
On the other hand, if $d(L_1, L_2) < k$, then $\sigma$ is not maximal: consider the line $L' = L_1+1$ and consider the intersecting vertex $u \in L'\cap L_3$.
Note that $d(u,v) \leq k$ for every $v\in \sigma$, contradicting that $\sigma$ is maximal.
We have proven the claim.

From the claim, note that $P$ is a diamond: a quadrilateral with sides of length $k$ and all sides parallel to the diagonal or the negative diagonal.
Let $c$ be the center of mass (intersecting point of the diagonals) of $P$ in $\R^2$.
Though $c$ need not be in $\Z^2$, each coordinate of $c$ is either an integer or a half-integer.
We note $P=B_{\R^2}[c,\frac{k}{2}]$.
It is clear that all pairs in $B_{\R^2}[c,\frac{k}{2}]$ are at distance at most $k$.
Also, since $\sigma \subseteq P = B_{\R^2}[c,\frac{k}{2}]$ and $\sigma$ is maximal, $\sigma$ must coincide with $B_{\R^2}[c,\frac{k}{2}] \cap \Z^2$.
\end{proof}

\begin{proposition}
\label{Proposition:MaximalSimplices}
Let $k\ge 0$ be an integer.
The only facets in $\vr{\Z^2}{k}$ are:
\begin{enumerate}
\item elements from $\Z ( \boldsymbol{\cdot} , k) \cup \Z (\square, k)$, if $k$ is even,
\item elements from $\Z ( \edge, k)$, if $k$ is odd.
\end{enumerate}
That is, the facets of $\vr{\Z^2}{k}$ are of the form $\sigma = B_{\R^2}[c,\frac{k}{2}] \cap \Z^2$ where
%$c$ is the center of mass of the generating set of $\sigma$, i.e.,
\begin{itemize}
\item $c=(i,j)$ or $(i+\frac{1}{2},j+\frac{1}{2})$ for $i,j\in \Z$ if $k$ is even, and
\item $c=(i+\frac{1}{2},j)$ or $(i,j+\frac{1}{2})$ for $i,j\in \Z$ if $k$ is odd.
\end{itemize}
\end{proposition}

\begin{proof}
Let $\sigma$ be a facet in $\vr{\Z^2}{k}$.
By Lemma~\ref{Lemma:EnclosedSigma}, we can write $\sigma = B_{\R^2}[c,\frac{k}{2}] \cap \Z^2$ with each coordinate of $c$ an integer of a half-integer.
Let $L_i$, for $1\leq i \leq 4$, denote the tangent lines that enclose $\sigma$ as in the proof of Lemma~\ref{Lemma:EnclosedSigma}.
That is, $L_1, L_2, L_3, L_4$ are of the form, $y=x+k_1$, $y=x+k_2$, $y=-x+k_3$, $y=-x+k_4$, respectively, so that they form a diamond with sides of length $k = d(L_1, L_2) = d(L_3, L_4)$.
%We consider the parity of $k$:

%Case 1: $k$ is even.
%If the top lines intersect at a point $L_1 \cap L_3\in Z^2$, then also $L_1 \cap L_4, L_2 \cap L_3, L_2\cap L_4 \in \Z^2$.
%Since $k$ is even, the center of this rhomboid is in $\Z^2$ and it forms the generating set of $\sigma$.
%Hence, $\sigma \in \Z ( \boldsymbol{\cdot} , k)$.
%Alternatively, if $L_1 \cap L_3 \notin \Z^2$, then since $k$ is even, we have $L_1 \cap L_4, L_2 \cap L_3, L_2 \cap L_4 \notin \Z^2$.
%There are horizontal top and bottom segments, and vertical left and right segments (four of the ``sides'' of $\sigma$) each containing only two points of $\sigma$.
%Now, consider the two vertical lines passing through the top and bottom vertices of $\sigma$, and the horizontal lines passing through the left and right vertices of $\sigma$.
%The intersection of these lines are the corners of a square which is the generating set of $\sigma$.
%Hence, $\sigma \in \Z (\square, k)$.

%Case 2: $k$ is odd.
%If $L_1 \cap L_3 \in \Z^2$, then also $L_2 \cap L_4 \in \Z^2$ (but $L_1 \cap L_4 ,L_2 \cap L_3\notin \Z^2$ since $k$ is odd).
%Consider the segment $S$ from $L_1 \cap L_3$ to $L_2 \cap L_4$.
%As $k$ is odd, the segment $S$ contains two central points in $\Z^2$, and these points are the generating set of $\sigma$.
%Hence $\sigma \in \Z ( \edge, k)$.
%If $L_1 \cap L_3 \notin \Z^2$, then $L_1 \cap L_4$ and $L_2 \cap L_3$ are in $\Z^2$.
%Symmetrically, the generating set of $\sigma$ is the pair of central vertices in the horizontal segment from $L_1 \cap L_4$ to $L_2 \cap L_3$.
%Hence $\sigma \in \Z ( \edge, k)$.

Since the top lines $L_1$ and $L_3$ are lattice diagonals (lines of slopes $\pm 1$ with integer $y$-intercepts), their intersection point $L_1\cap L_3$ has coordinates $(x,y)$ that are either both integers or both half-integers.
Thus $x+y$ is an integer.
The center $c$ of the ball with $\sigma=B_{\R^2}[c,\frac{k}{2}]\cap\Z^n$ is $\frac{k}{2}$ units below the point $L_1\cap L_3$, namely $c=(x,y-\frac{k}{2})$.
Therefore $c$ has integer or half-integer coordinates that add up to an integer if $k$ is even, and that add up to a half-integer if $k$ is odd.
\end{proof}

\begin{corollary}
\label{Corollary:NerveZ2}
%For $k\geq 0$, if $\sigma \in \vr{\Z^2}{k}$ is a maximal simplex, then $\sigma$ is isomorphic to the nerve $N(\{ B_{\R^2}(v,\frac{k+1}{2}) : v\in \sigma\})$.
%(Recall that balls, which are open sets in $\R^2$, do not contain their boundaries.
%Hence tangent balls have empty intersection.)
Let $k\ge 0$ be an integer.
Then $\sigma\subseteq \Z^2$ is a simplex in $\vr{\Z^2}{k}$ if and only if $\cap_{v\in\sigma}B_{\R^2}(v,\frac{k+1}{2}) \neq \emptyset$.
\end{corollary}

\begin{proof}
First we claim that for $c\in \R^2$ with each coordinate an integer or a half-integer, we have $B_{\R^2}[c,\frac{k}{2}] \cap \Z^2 = B_{\R^2}(c,\frac{k+1}{2}) \cap \Z^2$.
The $\subseteq$ direction is clear.
For the reverse direction, note if $z\in B_{\R^2}(c,\frac{k+1}{2}) \cap \Z^2$, then $d(z,c)<\frac{k+1}{2}$.
Since $z$ and $c$ have integer or half-integer coordinates, and since $k$ is an integer, this implies $d(z,c)\le\frac{k}{2}$, and hence $z\in B_{\R^2}[c,\frac{k}{2}]$.

Hence if $\sigma$ is a simplex in $\vr{\Z^2}{k}$, then by Proposition~\ref{Proposition:MaximalSimplices} and the claim above we have $\sigma = B_{\R^2}[c,\frac{k}{2}] \cap \Z^2 = B_{\R^2}(c,\frac{k+1}{2}) \cap \Z^2$, where $c$ has each coordinate an integer or a half-integer.
Hence $c\in \cap_{v\in\sigma}B_{\R^2}(v,\frac{k+1}{2})$, so $\cap_{v\in\sigma}B_{\R^2}(v,\frac{k+1}{2}) \neq \emptyset$.

Next we claim that if the intersection $\cap_{v\in\sigma}B_{\R^2}(v,\frac{k+1}{2})$ is nonempty for $\sigma\subseteq \Z^2$, then that intersection contains some point $c$ with each coordinate an integer or a half-integer.
The intersection of any open $l^1$ balls in $\R^2$ yields an intersection set equal to $\{(x,y)\in\R^2 : k_2 < y-x < k_1 ,\ k_4 < y+x < k_3\}$.
Write
\[ \cap_{v\in\sigma}B_{\R^2}(v,\tfrac{k+1}{2}) = \{(x,y)\in\R^2 : k_2 < y-x < k_1 ,\ k_4 < y+x < k_3\}.\]
Since the ball centers are in $\Z^2$, it follows that each of $k_1,k_2,k_3,k_4$ is an integer (if $k$ is odd) or an integer plus $\frac{1}{2}$ (if $k$ is even).
Since this intersection set is nonempty, we have $k_1-k_2\ge 1$ and $k_3-k_4\ge 1$.
It follows that the point $c=(x,y)\in\R^2$ with $y-x=k_1-\frac{1}{2}$ and $y+x=k_3-\frac{1}{2}$, namely the point $c=(\frac{k_3-k_1}{2},\frac{k_1+k_3-1}{2})$, is a point in $\cap_{v\in\sigma}B_{\R^2}(v,\frac{k+1}{2})$ with each coordinate an integer or a half-integer.

Hence if $\cap_{v\in\sigma}B_{\R^2}(v,\frac{k+1}{2}) \neq \emptyset$, then there is some $c$ in this intersection with each coordinate an integer or a half-integer.
Hence $\sigma \subseteq B_{\R^2}(c,\frac{k+1}{2})$.
Since $k$ is an integer and all coordinates of $c$ are integers or half-integers, this implies $\sigma \subseteq B_{\R^2}[c,\frac{k}{2}]$.
So $\diam(\sigma)\le k$ and $\sigma\in \vr{\Z^2}{k}$.
\end{proof}

This concludes the classification of the facets of $\vr{\Z^2}{k}$, which is the main point of this section.
As a consequence, we can deduce the known result that $\vr{\Z^2}{k}$ is contractible for all $k\ge 2$~\cite{zaremsky2024contractible,373007,wang2024contractibilityvietorisripscomplexesdense,virk2024contractibility}.

\begin{theorem}
\label{Theorem:MaximalSimplices}
$\vr{\Z^2}{k}$ is contractible for every $k\geq 2$, where $\Z^2$ is equipped with the $l^1$ metric.
\end{theorem}

\begin{proof}
By Corollary~\ref{Corollary:NerveZ2}, the $\vr{\Z^2}{k}$ is isomorphic to the nerve complex $N(\cU)$, where $\cU$ is the collection of open balls $\cU=\{B_{\R^2}(v,\frac{k+1}{2})\}_{v\in \Z^2}$.
Since $k\ge 2$, we have that $\cU$ is a cover of $\R^2$.
Since $l^1$ balls in $\R^2$ are convex, a finite intersection of such balls is convex, and hence either empty or contractible.
Therefore, the nerve theorem (Theorem~\ref{thm:nerve}) gives a homotopy equivalence $N(\cU) \simeq \cup\cU = \R^2$.
So $\vr{\Z^2}{k}\cong N(\cU) \simeq \R^2$ is contractible.
\end{proof}

\section{
Homotopy types of $\vr{T_{n,n}{k}}$.
%Complexes $\vr{T_{3k,3k}}{k}$ and $\vr{T_{3k-1,3k-1}}{k}$
}
\label{sec:wedgeS2S3}

In this section, we discuss the homotopy types of some of the complexes $\vr{T_{n,n}}{k}$.
We prove that $\vr{T_{n,n}}{k}$ is homotopy equivalent to the torus $\T^2$ for $k\ge 2$ and $n>3k$.
We show that for $k\geq 2$, $\vr{T_{3k,3k}}{k}$ is homotopy equivalent to a $(6k^2-1)$-fold wedge sum of $2$-spheres.
Finally, for $k\geq 3$, we prove $\vr{T_{3k-1,3k-1}}{k}$ is homotopy equivalent to a wedge sum of $6k-3$ many $S^2$'s and $6k-2$ many $S^3$'s.
In Section~\ref{ssec:maximal} we classify the facets in these simplicial complexes, before determining the homotopy types of the complexes in Section~\ref{ssec:homotopy}.

Though we expect some of the results in this section to have analogs for grid graphs in higher-dimensional tori, we do not pursue that direction here.

\subsection{Maximal simplices}
\label{ssec:maximal}

We denote $G_n=n\Z\times n\Z$ and its group action on the grid $\Z^2$ by $g\cdot (a, b)= (a+i, b+j)$ for $g=(i, j)\in G$.
Recall that we equip $\Z^2$ with the $l^1$ metric.
It is straightforward to verify that $T_{n, n}$ is isometric to $\Z^2/G_n$, equipped with the quotient metric, by the natural correspondence.
For convenience, we use $d$ to represent the metrics on $\Z^2$ and $\Z^2/G_n$.
Let $\pi_n$ be the quotient map from $\Z^2$ to $\Z^2/G_n$.
%For each $(a, b)\in \Z^2$, denote $([a], [b])=\pi_n(a, b)$.
For each $x\in \Z^2$, let $[x]=\{x'\in \Z^2~|~\pi_n(x')=\pi_n(x)\}$ be the corresponding element of $\Z^2/G_n = \pi_n(\Z^2)$.
If $x=(a,b)\in\Z^2$, then we may also write $[x]=([a],[b])$, where $[a]=\{a'\in \Z~|~a'=a\pmod n\}$.
%If any distinct pair $(a', b') \neq (a'', b'')$ satisfies $(a', b'), (a'', b'') \in ([a], [b])$, then $d((a', b'), (a'', b''))\geq n$.
%Notice that $d(([a], [b]), ([c], [d])) =\min \{d((a',b'), (c', d')): (a',b') \in ([a], [b])\text{ and }(c', d')\in ([c], [d]) \}$.
If any distinct pair $x' \neq x''\in \Z^2$ satisfies $x', x'' \in [x]$, then $d(x', x'')\geq n$ since either the horizontal or the vertical coordinates differ by at least $n$.
Notice that $d([x], [y]) =\min \{d(x',y')~|~ x' \in [x]\text{ and }y'\in [y] \}$.
Using these facts, the following result is straightforward to verify.

\begin{lemma}
\label{lemma:unique_pre}
%Fix $n>2k$.
%Suppose that $d(([a], [b]), ([c], [d]))\leq k$.
%Then for any $(a', b')\in ([a], [b])$, there exists a unique $(c', d')\in ([c], [d])$ such that
%\[d((a',b'), (c', d'))=d(([a], [b]), ([c], [d]) )\leq k.\]
Fix $n>2k$ and $[x],[y]\in \pi_n(\Z^2)$.
Suppose that $d([x],[y])\leq k$.
Then for any $x'\in [x]$, there exists a unique $y'\in [y]$ such that
\[d(x',y')=d([x],[y])\leq k.\]
\end{lemma}

Next we discuss the relations between the facets in $\vr{\Z^2}{k}$ and $\vr{\pi_n(\Z^2)}{k}$.

\begin{lemma}
\label{lemma:imageoffacets}
Let $\sigma$ be a facet in the complex $\vr{\Z^2}{k}$.
For $n>2k+1$, $\pi_n(\sigma)$ is a facet in $\vr{\pi_n(\Z^2)}{k}$.
\end{lemma}

\begin{proof}
%We let $x$ be an element in $\Z^2$ and $[x]=\pi_n(x)$.
Clearly $\pi_n(\sigma)$ is a simplex in $\vr{\pi_n(\Z^2)}{k}$.
Suppose that $\pi_n(\sigma)$ is not maximal.
Then pick $[y]\notin \pi_n(\sigma)$ such that $d([y], [x])\leq k$ for all $x\in \sigma$.
By Lemma~\ref{lemma:unique_pre}, for each $x\in \sigma$, there is a unique $y_x\in [y]$ such that $d(x, y_x)
\leq k$.
If all of the $y_x$ points are the same, then $\sigma\cup \{y_x\}$ is a simplex in $\vr{\Z^2}{k}$, which contradicts with the maximality of $\sigma$.
Otherwise by the shape of $\sigma$ as in Proposition~\ref{Proposition:MaximalSimplices}, there are $x$ and $x'$ in $\sigma$ such that $d(x, x')=1$ and $y_x\neq y_{x'}$.
Then $d(y_x, y_{x'})\leq d(y_x, x)+d(x, x') + d(x', y_{x'}) \leq 2k+1$, which contradicts the fact that $d(y_x, y_{x'})\geq n>2k+1$.
\end{proof}

\begin{lemma}
\label{lemma:prefacets}
Suppose that $n\geq 3k-1$ and $k\geq 1$.
Let $x,y\in\Z^2$ with $x\neq y$, with $3k-n< d(x,y)\leq k$, and with $\pi_n(B_{\Z^2}[x, k]\cap B_{\Z^2}[y, k]) =\pi_n(B_{\Z^2}[x, k])\cap \pi_n(B_{\Z^2}[y, k])$.
Then any facet $\tau$ in $\vr{\pi_n(\Z^2)}{k}$ containing $[x]$ and $[y]$ is of the form $\pi_n(\sigma)$ for some facet $\sigma\in \vr{\Z^2}{k}$.
\end{lemma}

\begin{proof}
Let $B[x,k]=B_{\Z^2}[x,k]$.
Fix a facet $\tau\in\vr{\pi_n(\Z^2)}{k}$ containing $[x]$ and $[y]$.
By Lemma~\ref{lemma:unique_pre}, $\pi_n(B[x, k])=B_{\pi_n(\Z^2)}[[x], k]$ for each $x\in \Z^2$ since $n>2k$.
Hence, $\tau\subseteq\pi_n(B[x, k])\cap \pi_n(B[y, k])=\pi_n(B[x, k]\cap B[y, k])$.

For each $[z]\in \tau$, Lemma~\ref{lemma:unique_pre} guarantees a unique representative $z_x$ in $[z]$ such that $d(x, z_x)\leq k$.
Note that $y_x=y$ since $1\leq d(x,y)\leq k$.
Define $\sigma=\{x\}\cup \{z_x: [z]\in \tau \text{ and } z\neq x\}$.
It is clear that $\tau=\pi_n(\sigma)$ and that $\sigma\subseteq B[x,k]$.
We will show that $\sigma$ is a facet in $\vr{\Z^2}{k}$.

First, we show that $\sigma\subseteq B[x,k]$ furthermore satisfies $\sigma\subseteq B[x, k]\cap B[y, k]$.
Let $z_x$ be an arbitrary element of $\sigma$, with $z_x\in [z]$ for some $[z]\in \tau \subseteq \pi_n(B[x, k]\cap B[y, k])$.
So there must be some $z'\in [z]$ with $z'\in B[x, k]\cap B[y, k]$. By Lemma~\ref{lemma:unique_pre}, $z_x=z'$.
Hence $z_x\in B[x, k]\cap B[y, k]$, as desired.

We now show that $\sigma$ is a simplex in $\vr{\Z^2}{k}$.
Let $[z]$ and $[z']$ be arbitrary elements of $\tau$, meaning that $z_x$ and $z'_x$ are arbitrary elements of $\sigma$.
Since $\tau\in\vr{\pi_n(\Z^2)}{k}$, we have $d([z], [z'])\leq k$.
Clearly $d(z_x, z_x')\leq 2k$. 
We claim that for any $z''\in[z'_x]\setminus\{z'_x\}$, we have $d(z_x,z'')\ge k+1$.
There are two cases.
Suppose first that $z_x$ and $z_x'$ have same first or second coordinates.
Since $z_x, z'_x\in\sigma \subseteq B[x, k]\cap B[y, k]$ have the same first or second coordinates, elementary geometry gives $d(z_x, z'_x)\leq 2k-d(x,y)$.
For any element $z''\in [z'_x]\setminus\{z'_x\}$, the inequality $d(z'_x, z'')\geq n$ with the triangle inequality give $d(z_x, z'')\geq n-2k+d(x,y)>k$.
For the second case, suppose that $z_x$ and $z_x'$ are not in the same vertical or horizontal line. Let  $z_x=(z_1, z_2)$ and $z_x'=(z_1', z_2')$.
Then $1\leq |z_1-z_1'|\leq 2k-1$ and also $1\leq |z_2-z_2'|\leq  2k-1$.
Pick $z''\in [z'_x]\setminus\{z'_x\}$. Without loss of generality, suppose $z''_2\neq z'_2$.
Then by triangle inequality,  \[|z_2-z_2''|\geq |z_2'-z_2''|-|z_2-z_2'|\geq n-(2k-1)\geq k.\] 
This shows that $d(z_x, z'')= |z_1-z_1'|+|z_2-z_2''|\geq 1+k$.
Lemma~\ref{lemma:unique_pre} therefore implies that $d(z_x, z'_x)= d([z], [z'])\le k$, and so $\sigma \in \vr{\Z^2}{k}$.

Lastly, we show that $\sigma$ is maximal.
Since $n> 3k-d(x,y) \ge 2k$, Lemma~\ref{lemma:unique_pre} implies that $\sigma=\{x\}\cup \{z_x: [z]\in \tau \text{ and } z\neq x\}$ is a maximal simplex in $\vr{\Z^2}{k}$.
\end{proof}

\begin{lemma}
\label{lemma:emptyintersection} 
Fix different $x=(x_1,x_2)$ and $y=(y_1,y_2)$ in $\Z^2$.
If $|x_1-y_1|<n-2k$ and $|x_2-y_2|<n-2k$, then $\pi_n (B_{\Z^2}[x,k]\cap B_{\Z^2}[y,k]) =\pi_n(B_{\Z^2}[x,k]) \cap \pi_n(B_{\Z^2}[y,k])$.
\end{lemma}

\begin{proof}
The assumptions imply $n>2k$.
Let $B[x,k]=B_{\Z^2}[x,k]$.
Let $[z]\in \pi_n(B[x,k]\cap B[y,k])$.
Then there exists some $z'\in[z]$ with $z'\in B[x,k]\cap B[y,k]$.
Since $z'\in B[x,k]$, then $[z]\in \pi_n(B[x,k])$.
Similarly, $[z]\in \pi_n(B[y,k])$.
So $\pi_n(B[x,k]\cap B[y,k])\subseteq\pi_n(B[x,k]) \cap \pi_n(B[y,k])$.

Next, let $[z]\in\pi_n(B[x,k]) \cap \pi_n(B[y,k])=B_{\pi_n(\Z^2)}[[x],k]\cap B_{\pi_n(\Z^2)}[[y],k]$.
By Lemma~\ref{lemma:unique_pre}, there exist unique elements $z_x, z_y\in[z]$ such that $d(z_x, x)\leq k$ and $d(z_y,y)\leq k$.
The horizontal coordinates of $z_x$ and $z_y$ must be equal, since otherwise they would differ by at least $n$, which would give that the horizontal coordinates of $x$ and $y$ would differ by at least $n-2k$, a contradiction.
Similarly, the vertical coordinates of $z_x$ and $z_y$ must be equal.
Hence $z_x=z_y\in B[x,k]\cap B[y,k]$, giving $[z]\in \pi_n(B[x,k]\cap B[y,k])$.
So $\pi_n(B[x,k]) \cap \pi_n(B[y,k])\subseteq \pi_n (B[x,k]\cap B[y,k])$, as desired.
\end{proof}

For a simplicial complex $K$, we let $M(K)$ denote the collection of facets in $K$.
By Lemma~\ref{lemma:prefacets}, every facet $\tau\in\vr{\pi_n(\Z^2)}{k}$ is either the image of a facet in $\vr{\Z^2}{k}$, or there exists a pair of distinct vertices $x=(x_1,x_2)$, $y=(y_1,y_2)\in\tau$ such that
\begin{equation}
\label{eq:intersection}
\pi_n(B_{\Z^2}[x, k]\cap B_{\Z^2}[y, k]) \neq \pi_n(B_{\Z^2}[x, k])\cap \pi_n(B_{\Z^2}[y, k]),
\end{equation}
or every pair of vertices in $\tau$ has distance $\leq 3k-n$.
By Lemma~\ref{lemma:emptyintersection}, the equation~\eqref{eq:intersection} implies that $n-2k\leq |x_i-y_i|\leq k$ for either $i=1$ or $i=2$.

When $n=3k$, this forces $|x_i-y_i|=k$, meaning that two vertices must differ by $k$ in one coordinate.
It is easy to verify that the only simplices in this case are those consisting of three vertices equally spaced by $k$ along a row or column of $T_{3k,3k}$; explicitly, these are of the form 
\[\{([a], [b]), ([a+k], [b]), ([a+2k], [b])\}\text{ or }\{([a], [b]), ([a], [b+k]), ([a], [b+2k])\},\]
with all coordinates taken modulo $3k$.
These simplices are maximal in $\vr{T_{3k,3k}}{k}$ and their projections on the non-constant coordinate are maximal in $\vr{C_{3k}}{k}$.

Similarly, when $n=3k-1$, the inequality becomes $k-1\leq |x_i-y_i|\leq k$, and the non-liftable facets in this case are the tetrahedra
\begin{center}
$\{([a], [b]), ([a+k], [b]), ([a+2k-1], [b]), ([a+2k], [b])\}$\\\text{ or }\\$\{([a], [b]), ([a], [b+k]), ([a], [b+2k-1]), ([a], [b+2k])\}$.
\end{center}
These tetrahedra are also maximal in $\vr{T_{3k-1,3k-1}}{k}$ and their projections on the non-constant coordinate are maximal in $\vr{C_{3k-1}}{k}$.
For a proof of these results, we direct the reader to Lemma~\ref{lemma:appendix}.
%\note{Henry says: When do we use the %notation $\cl{(C_{n}^{k})}$ and when do we %use the notation $\vr{C_{n}}{k}$?
%Is one the primary notation that we plan %to use the most?}
%\note{I have been using them %interchangeably as we note their %equivalence in the graphs preliminary %section. However, for consistency, let's %use $\vr{C_{n}}{k}$.}

Let $M_{n, k}$ denote the collection of maximal simplices in $\vr{T_{n, n}}{k}$ inherited from $\vr{\Z^2}{k}$, i.e.\ 
\[M_{n, k} = \{\pi_n (\sigma): \sigma \text{ is a facet in }\vr{\Z^2}{k}\}.\] 

Additionally, let $N_{3k,k}$ be the collection of triangles in the quotient space, i.e.\ 
\[N_{3k,k}=\left\{\substack{\{([a], [b]), ([a+k], [b]), ([a+2k], [b])\} \\ \{([a], [b]), ([a], [b+k]), ([a], [b+2k])\}}~|~0\leq a, b \le 3k-1\right\}.\]

And we let $N_{3k-1,k}$ be the collection of tetrahedra in the quotient space, i.e.\ 
\[N_{3k-1,k}=\left\{\substack{\{([a], [b]), ([a+k], [b]), ([a+2k-1], [b]), ([a+2k], [b])\} \\ \{([a], [b]), ([a], [b+k]), ([a], [b+2k-1]), ([a], [b+2k])\}}~|~0\leq a, b \leq 3k-2\right\}.\]

\begin{figure}[h]
\begin{center}
\begin{tikzpicture}
\def \n {8} 
\def \radius {1.6cm} 

\foreach \s in {1,...,\n}
{
  \node[draw, circle, fill=black, inner sep=2pt, minimum size=8pt] 
    (v\s) at ({360/\n * (\s - 1)}:\radius) {};

  %labels
  \node[font=\small] at ({360/\n * (\s - 1)}:\radius + 0.5cm) {$v_{\s}$};
 }

 %edges
 \foreach \s in {1,...,\n}
 {
  \pgfmathtruncatemacro{\next}{mod(\s, \n) + 1}
  \draw (v\s) -- (v\next);
 }

 \draw[blue, thick] (v6) -- (v7);
 \draw[blue, thick] (v7) -- (v1);
 \draw[blue, thick] (v4) -- (v1);
 \draw[blue, thick] (v4) -- (v6);
 \draw[blue, thick] (v4) -- (v7);
 \draw[blue, thick] (v1) -- (v6);

\begin{scope}[xshift=5.0cm]
\def \n {9} 
\def \radius {1.6cm} 

\foreach \s in {1,...,\n}
{
  \node[draw, circle, fill=black, inner sep=2pt, minimum size=8pt] 
    (v\s) at ({360/\n * (\s - 1)}:\radius) {};

  %labels
  \node[font=\small] at ({360/\n * (\s - 1)}:\radius + 0.5cm) {$v_{\s}$};
 }

 %edges
 \foreach \s in {1,...,\n}
 {
  \pgfmathtruncatemacro{\next}{mod(\s, \n) + 1}
  \draw (v\s) -- (v\next);
 }

\draw[blue, thick] (v1) -- (v7);
 \draw[blue, thick] (v4) -- (v1);
 \draw[blue, thick] (v4) -- (v7);

\end{scope}

\end{tikzpicture}
\caption{Example simplices in $N_{8,3}$ and $N_{9,3}$}
\end{center}
\end{figure}

\begin{lemma}
\label{lem:facets_3kor3k-1}
For any $k\geq 2$, the collection of facets in $\vr{T_{3k,3k}}{k}$ is $M_{3k,k}\cup N_{3k,k}$.
Also, for any $k\geq 3$, the collection of facets in $\vr{T_{3k-1,3k-1}}{k}$ is $M_{3k-1,k}\cup N_{3k-1,k}$.
\end{lemma}

\begin{proof}
We have $M_{3k,k}\subseteq M(\vr{T_{3k,3k}}{k})$ by Lemma~\ref{lemma:imageoffacets} (with $3k=n>2k+1$ since $k\ge 2$).
Also, we have the containment $N_{3k,k}\subseteq M(\vr{T_{3k,3k}}{k})$.
Similarly, we have $M_{3k-1,k}\subseteq M(\vr{T_{3k-1,3k-1}}{k})$ by Lemma~\ref{lemma:imageoffacets} (with $3k-1=n>2k+1$ since $k\ge 3$), along with the containment $N_{3k-1,k}\subseteq M(\vr{T_{3k-1,3k-1}}{k})$.

Next we pick a facet $\tau$ in the complex $\vr{T_{3k,3k}}{k}$.
We'll show that $\tau$ is in $M_{3k,k}\cup N_{3k,k}$.

First we suppose that there exist $[x] =([a], [b])$ and $[y]=([c], [d])$ in $\tau$ with $[a]\neq [c]$ and $[b]\neq [d]$.
Without loss of generality, we assume that $x=(a, b)$ and $y=(c, d)$ and $d(x, y)=d([x], [y])\le k$.
Clearly $a\neq c$ and $b\neq d$; hence $1\leq d(a, c)\leq k-1$ and $1\leq d(b, d)\leq k-1$.
By Lemma~\ref{lemma:emptyintersection} (with $n-2k=k$), we have $\pi_n (B_{\Z^2}[x, k]\cap B_{\Z^2}[y, k]) =\pi_n(B_{\Z^2}[x, k]) \cap \pi_n(B_{\Z^2}[y, k])$.
Then by Lemma~\ref{lemma:prefacets}, we have $\tau=\pi_n(\sigma)$ for some facet $\sigma$ in $\vr{\Z^2}{k}$.
Therefore $\tau\in M_{3k,k}$.

%Then, we suppose that any pair of elements in $\tau$ share the same first coordinate or second coordinate.
%This implies that all elements in $\tau$ have the same first coordinate or second coordinate.
%We assume thatthe first coordinate of all elements in $\tau$ is $[a]$.

Otherwise, all elements in $\tau$ have the same first coordinate or second coordinate.
Without loss of generality we assume that the first coordinate of all elements in $\tau$ is $[a]$.
The full subcomplex of $\vr{T_{3k,3k}}{k}$ whose vertices have first coordinate $[a]$ is isomorphic to the clique complex $\vr{C_{3k}}{k}$, where $C_{3k}$ is the cyclic graph with $3k$ vertices.
Of the maximal simplices in this copy of $\vr{C_{3k}}{k}$, only the ones of the form $\{([a], [b]), ([a], [b+k]), ([a], [b+2k])\}$ for some $0\leq b\leq 3k-1$ are maximal in $\vr{T_{3k,3k}}{k}$.
Hence $\tau\in N_{3k,k}$.
This completes the proof of the first statement of the lemma. 

\medskip 

For the second statement, it remains to show that a facet $\tau$ in $\vr{T_{3k-1,3k-1}}{k}$ is in $M_{3k-1,k}\cup N_{3k-1,k}$.
Consider the case where there exist $[x] =([a], [b])$ and $[y]=([c], [d])$ in $\tau$ with $[a]\neq [c]$ and $[b]\neq [d]$.
Then the same argument as before applies, \emph{except} we cannot apply Lemma~\ref{lemma:emptyintersection} when $d(x, y)=k$ with either $d(a, c)=1$ or $d(b, d)=1$.
Without loss of generality, assume $d(a, c)=1$ and $d(b, d)=k-1$.
We show that still in this case $\pi_n (B_{\Z^2}[x, k]\cap B_{\Z^2}[y, k]) =\pi_n(B_{\Z^2}[x, k]) \cap \pi_n(B_{\Z^2}[y, k])$, in which case $\tau$ is in $M_{3k-1,k}$ by Lemma~\ref{lemma:prefacets}.
Moreover, we assume that $a<c$ and $b<d$, and then $c=a+1$ and $d=b+k-1$.
Other cases can be proved similarly.
Pick an arbitrary $[z]\in B_{\Z^2}[y, k]$ with $d((z_1, z_2), y)=|z_1-c|+|z_2-d|\leq k$. 
It is sufficient to show that $d(z', x)>k$ for any $z'\in [z]$ with $z'\neq (z_1, z_2)$.
We show this holds for $z'=(z_1+3k-1, z_2)$, and similar approach can be applied for $z'$ in different forms.
Notice that \[d(z', x)=|z_1+3k-1-a|+|z_2-b|=|z_1-c+3k|+|z_2-d+k-1|\geq 2k.\]   
%\note{Henry asks: If I want to try to prove Lemmas~\ref{lem:facets_3kor3k-1} and~\ref{lem:facets_3kor3k-1-1} at the same time, then is it correct to say that this paragraph changes based on the maximal simplices in the cyclic graph, and that the above paragraph doesn't change much and still follows from all the earlier lemmas in this section? If so, I may try this out.}
%\note{Ziqin says: Yes. I believe 5.5 and 5.6 follow from the previous lemmas and the clarifications of facets in the clique power of cyclic graphs.} 
%\note{John says: In the $n=3k-1$ case, I can see how, in the case that all elements in $\tau$ have the same first or second coordinate, then $\tau \in N_{3k-1,k}$. However, I'm having some difficulty in seeing how this would imply $\tau\in M_{3k-1,k}$ if $\tau$ contains $[x] =([a], [b])$ and $[y]=([c], [d])$ such that $[a]\neq [c]$ and $[b]\neq [d]$.  How would you apply Lemma 5.4 as $n-2k=k-1$? In this case, the inequalities $1\leq|a-c|\leq k-1$ and $1\leq|b-d|\leq k-1$ do not necessarily guarantee that $|a-c|<n-2k=k-1$ and $|b-d|<n-2k=k-1$. If the Lemma 5.4 conditions were instead $|x_1-y_1|\leq n-2k$ and $|x_2-y_2|\leq n-2k$, then I believe it would work.} \note{Henry says: John, good question; I am not sure.}

In the remaining case when all elements of $\tau$ have the same first coordinate or second coordinate, assume again that the first coordinate of all elements in $\tau$ is $[a]$.
The full subcomplex of $\vr{T_{3k-1,3k-1}}{k}$ whose vertices have first coordinate $[a]$ is isomorphic to the clique complex $\vr{C_{3k-1}}{k}$.
Of the maximal simplices in this copy of $\vr{C_{3k}}{k}$, only the ones in $N_{3k-1,k}$ are maximal in $\vr{T_{3k-1,3k-1}}{k}$.
\end{proof}

%Similarly, we can prove the following.

%\begin{lemma} \label{lem:facets_3kor3k-1-1} For any $k\geq 3$, the collection of facets in $\vr{T_{3k-1,3k-1}}{k}$ is $M_{3k-1,k}\cup N_{3k-1,k}$. \end{lemma}

%\begin{proof} \note{Ziqin says: I'm trying to write out a proof first here and then combine 5.5 and 5.6 together.} \end{proof}

\begin{lemma}
\label{lem:facets_n>3k}
For $n> 3k$, the collection of facets in $\vr{T_{n, n}}{k}$ is $M_{n, k}$.
\end{lemma}

\subsection{Homotopy types}
\label{ssec:homotopy}

In Theorem~\ref{thm:homotopy_3k} below we will prove that for any $k\geq 2$, we have the homotopy equivalence $\vr{T_{3k,3k}}{k}\simeq \bigvee_{6k^2-1} S^2$.
The geometric flavor of our proof is as follows.
Recall from Lemma~\ref{lem:facets_3kor3k-1} that the set of facets in $\vr{T_{3k,3k}}{k}$ is $M_{3k,k}\cup N_{3k,k}$.
By Theorem~\ref{thm:torus_homotopy}, the subcomplex generated by $M_{3k,k}$ is homotopy equivalent to a torus.
To this torus, we glue on the $6k$ triangles in $N_{3k,k}$ along their boundaries.
Of these, $3k$ triangles have their boundaries wrapped once around the longitudinal circle, and the remaining $3k$ triangles have their boundaries wrapped once around the meridional circle.
To understand the resulting homotopy type, we note (see Lemma~\ref{lemma:subcomplex_3k}) that gluing on only one longitudinal triangle and only one meridional triangle produces the $2$-sphere.
Then, gluing on each of the remaining $6k-2$ triangles adds an additional $2$-sphere to the wedge summand, producing $\bigvee_{6k^2-1} S^2$ in total.

In order to complete the inductive format of this proof, we will need the following result describing the homotopy type of a complex by splitting it into two subcomplexes, with a proof in~\cite{GSS22}.

\begin{lemma}
\label{cup_simp}
Suppose a simplicial complex $K=K_1\cup K_2$ satisfies that the inclusion maps $\imath_1: K_1\cap K_2\rightarrow K_1$ and $\imath_2: K_1\cap K_2\rightarrow K_2$ are both null-homotopic.
Then 
\[K\simeq K_1\vee K_2\vee\Sigma(K_1\cap K_2).\]
\end{lemma}

\begin{theorem}
\label{thm:torus_homotopy}
Assume that $k\geq 2$.
For any $n> 2k$, let $K$ be the complex with $M_{n, k}$ being the collection of maximal simplices.
Then $K$ is homotopy equivalent to the torus $\T^2$.

Hence when $k\ge 2$ and $n>3k$, $\vr{T_{n, n}}{k}$ is homotopy equivalent to the torus $\T^2$.
\end{theorem}

This theorem therefore gives the torus homotopy types in Table~\ref{table:Torus-Grids} for $n>3k$.

\begin{proof}
The topological space $\R^2/G_n$ is homeomorphic to the torus $\T^2$.
Similarly as in the proof of Theorem~\ref{Theorem:MaximalSimplices}, we define $\cU= \{\pi_n(B_{\R^2}(x,\frac{k+1}{2}))\subseteq \R^2/G_n : x\in\Z^2\}$, where here $\pi_n \colon \R^2 \to \R^2/G_n$.
Since $n>2k$ and since $k\geq 2$, the intersection of any finite subcollection of $\cU$ is either empty or contractible by Lemma~\ref{lemma:emptyintersection}.
Also, the nerve complex $N(\cU)$ is isomorphic to $K$.
By the nerve theorem, $N(\cU)$ is homotopy equivalent to $\T^2$, and hence so is $K$.

When $n>3k$, the collection of facets in the complex $\vr{T_{n, n}}{k}$ is $M_{n, k} $ by Lemma~\ref{lem:facets_n>3k}; hence in this case $\vr{T_{n, n}}{k}$ is homotopy equivalent to $\T^2$.
\end{proof}

\begin{lemma}
\label{lemma:subcomplex_3k}
Assume that $k\geq 2$.
Let $\tau_1=\{([0], [0]), ([k], [0]), ([2k], [0]) \}$ and $\tau_2=\{([0], [0]), ([0], [k]), ([0], [2k])\}$ and let $K$ be the simplicial complex with $M(K) = M_{3k,k}\cup \{\tau_1, \tau_2\}$.
Then $K$ is homotopy equivalent to the sphere $S^2$.
\end{lemma}

\begin{proof}
Let $X$ be the cell complex obtained from $\R^2/G_{3k}$ by gluing two disks $c_1$ and $c_2$ along the circles $(\R\times \{0\})/G^1_{3k}$ and and $(\{0\}\times \R)/G^2_{3k}$,  where $G^1_{3k} = 3k\Z\times\{0\}$ and $G^2_{3k}= \{0\}\times 3k\Z$.
For convenience, let $X_1$ be the torus $\R^2/G_{3k}$ and let $X_2$ be the image of $c_1, c_2$ under the gluing map; hence, $X=X_1\cup X_2$.
Since $X_2$ is contractible, $X$ is homotopy equivalent to its quotient over $X_2$ which is a cell complex generated by one vertex and one disk; and hence it is homotopy equivalent to $S^2$.

Let $K_1$ be the complex whose collection of facets is $M_{3k,k}$ and let $K_2$ be the complex whose collection of facets is $\{\tau_1, \tau_2\}$.
Clearly the geometric realization of $K_2$ is homeomorphic to $X_2$.
By Theorem~\ref{thm:torus_homotopy}, $K_1$ is homotopy equivalent to $X_1$ because both are homotopy equivalent to $\T^2$.
Denote $L= K_1\cap K_2$ and $A = X_1\cap X_2$.
Both $L$ and $A$ are homotopy equivalent to $S^1\vee S^1$.
Fix a homotopy equivalence $g_1$ from $X_1$ to $K_1$ and a homeomorphism $g_2$ from $X_2$ to $K_2$ such that the restrictions of these two maps are homeomorphisms from $A$ to $L$.
Therefore, the below diagram (where all $i_1, \ldots, i_4$ are inclusion maps) commutes.

\begin{center}
\begin{tikzpicture}[scale=1.2, every node/.style={scale=0.9}]
    % First row
    \node (X) at (-3, 2)  {\( X_1 \)};
    \node (A) at (0, 2)  {\( A \)};
    \node (Y) at (3, 2) {\( X_2 \)};
    
    % Arrows in first row
    \draw[->, thick] (A) -- (X) node[midway, above] {\(i_1\)};
    \draw[<-, thick] (Y) -- (A) node[midway, above] {\(i_2\)};
    
    % Second row
    \node (HX) at (-3, 0.5)  {\( K_1 \)};
    \node (HA) at (0, 0.5) {\(L \)};
    \node (HY) at (3, 0.5) {\( K_2 \)};
    
    % Labels for second row
    %\node at (-3, -1) {\( X' \)};
    %\node at (3, -1) {\( Y' \)};
    %node at (0, -1) {\( A' \)};
    
    % Arrows in second row
    \draw[<-, thick] (HX) -- (HA) node[midway, below] {\(i_3\)};
    \draw[<-, thick] (HY) -- (HA) node[midway, below] {\(i_4\)};

    \draw[->, thick] (X) -- (HX) node[midway, left] {\(g_1\)};
    \draw[->, thick] (A) -- (HA) node[midway, left] {\(h\)};
     \draw[->, thick] (Y) -- (HY) node[midway, left] {\(g_2\)};
    
    % Connecting rows with equivalence symbols
    %\node at (-3, 1) {\(\simeq\)};
   % \node at (0, 1) {\(\simeq\)};
    %\node at (3, 1) {\(\simeq\)};
\end{tikzpicture}
\end{center}
Because $X=X_1\cup X_2$ is homotopic to $S^2$, by~\cite[7.5.7]{Brown} so is $K=K_1\cup K_2$.
This finishes the proof.
\end{proof}

\begin{theorem}
\label{thm:homotopy_3k}
For $k\geq 2$, we have the homotopy equivalence
\[\vr{T_{3k,3k}}{k}\simeq \bigvee_{6k^2-1} S^2.\]
\end{theorem}

\begin{proof}
There are $6k^2$ distinct $2$-simplices in $N_{3k,k}$ and the intersection of each pair in $N_{3k,k}$ is at most a vertex.
Let $K$ be the complex whose facets are
\[M(K) = M_{3k,k}\cup \{\{([0], [0]), ([k], [0]), ([2k], [0]) \}, \{([0], [0]), ([0], [k]), ([0], [2k])\}\}.\]
Then by Lemma~\ref{lemma:subcomplex_3k}, $K$ is homotopy equivalent to a sphere.
List the rest of the elements in $N_{3k,k}$ as $\sigma_1, \sigma_2, \ldots, \sigma_{6k^2-2}$.

Let $K_\ell$ be the complex formed by including the additional simplices $\sigma_1, \ldots, \sigma_\ell$.
Then we claim that the complex $K_\ell$ is homotopy equivalent to a $(\ell+1)$-fold wedge sum of $S^2$'s.
This implies that $\vr{T_{3k,3k}}{k}$ is homotopy equivalent to $\bigvee_{6k^2-1} S^2$.

The claim clearly holds for $\ell=0$.
Now we assume it holds for some $\ell<6k^2-2$.
Let $K_{\sigma_{\ell+1}}$ be the complex generated by the simplex $\sigma_{\ell+1}$.
Hence, $K_{\ell+1} =K_\ell\cup K_{\sigma_{\ell+1}}$.
Note that $K_{\sigma_{\ell+1}}$ is a complex formed by a $2$-simplex whose boundary is in $K_\ell$; hence, $K_\ell\cap K_{\sigma_{\ell+1}}$ is homotopy equivalent to $S^1$, and hence null-homotopic in both $K_\ell$ and $K_{\sigma_{\ell+1}}$.
Then by Lemma~\ref{cup_simp}, $K_{\ell+1} =K_\ell\cup K_{\sigma_{\ell+1}}$ is homotopy equivalent to $K_\ell\vee \Sigma S^1 = K_\ell\vee S^2$ since $K_{\sigma_{\ell+1}}$ is contractible.
This induction finishes the proof of the claim.
\end{proof}

This gives the homotopy types of the \colorbox{purple!60}{red} diagonal in Table~\ref{table:Torus-Grids}.
We now prove the homotopy types of the \colorbox{teal!50}{teal} diagonal.

In Theorem~\ref{thm:homotopy_3k-1} below we prove that for any $k\geq 3$, we have the homotopy equivalence $\vr{T_{3k-1,3k-1}}{k}\simeq \bigvee_{6k-3} S^2\vee \bigvee_{6k-2}S^3$.
Recall from Lemma~\ref{lem:facets_3kor3k-1} that the facets in $\vr{T_{3k-1,3k-1}}{k}$ are $M_{3k-1,k}\cup N_{3k-1,k}$.
By Theorem~\ref{thm:torus_homotopy}, the subcomplex generated by $M_{3k-1,k}$ is homotopy equivalent to a torus.
To this torus, we glue on the remaining tetrahedra in $N_{3k-1,k}$, which as we show, have the effect of gluing on $3k-1$ $3$-spheres attached along longitudinal circles, and $3k-1$ $3$-spheres attached along meridional circles.
The resulting homotopy type, as we show, will be $\bigvee_{6k-3} S^2\vee \bigvee_{6k-2}S^3$.

We omit the proof of the following lemma, which is analogous to Lemma~\ref{lemma:subcomplex_3k}.

\begin{lemma}
\label{lemma:subcomplex_3k-1}
Assume that $k\geq 3$.
Let $\tau_1=\{([0], [0]), ([k], [0]), ([2k-1], [0]), ([2k], [0]) \}$ and $\tau_2=\{([0], [0]), ([0], [k]), ([0], [2k-1]), ([0], [2k])\}$ and let $K$ be the simplicial complex with $M(K) = M_{3k-1,k}\cup \{\tau_1, \tau_2\}$.
Then $K$ is homotopy equivalent to the sphere $S^2$.
\end{lemma}

In Theorem~\ref{thm:homotopy_3k-1}, we use a result of Adamaszek~\cite[Corollary 6.7]{Adamaszek2013} to determine the homotopy types of the cross-sectional full sub-complexes in $\vr{T_{3k-1,3k-1}}{k}$.
This result characterizes the homotopy type of the Vietoris-Rips complex of a cycle:
\begin{align*}
\vr{C_n}{k} \cong 
\begin{cases}
\bigvee^{n-2k-1}S^{2\ell} & \text{if } k =\frac{\ell}{2\ell+1}n \\
S^{2\ell+1} &\text{if } \frac{\ell}{2\ell+1}n<k<\frac{\ell+1}{2\ell+3}n.
\end{cases}
\end{align*}
In particular, when $k\geq 3$, $\vr{C_{3k-1}}{k}$ is homotopy equivalent to $S^3$.

\begin{theorem}
\label{thm:homotopy_3k-1}
For $k\geq 3$, we have the homotopy equivalence
\[\vr{T_{3k-1,3k-1}}{k}\simeq \bigvee_{6k-3} S^2\vee \bigvee_{6k-2}S^3.\] 

Also, \[\vr{T_{5,5}}{2} \simeq \bigvee_{9}S^2.\] 
\label{thm:S2S3}
\end{theorem}

\begin{proof}
For each $a\in \Z$, we let $L_a^1$ be the full subcomplex of $\vr{T_{3k-1,3k-1}}{k}$ with the vertices whose first entry is $[a]$, and we let $L_b^2$ be full subcomplex of $\vr{T_{3k-1,3k-1}}{k}$ with all vertices whose second entry is $[b]$.
Notice that both $L_a^1$ and $L_b^2$ are isomorphic to the Vietoris--Rips complex on the cyclic graph $C_{3k-1}$ with scale $k$.

First fix $k\geq 3$.
Then, by Adamaszek's result, both of $L_a^1$ and $L_b^2$  are homotopy equivalent to $S^3$.
Notice that the pairwise intersection of these subcomplexes contains at most one vertex.
We let $K$ be the simplicial complex whose collection of facets is $M_{3k-1,k}\cup \{\sigma_1, \sigma_2\}$, where $\sigma_1=\{([0], [0]), ([0], [k]), ([0], [2k-1]), ([0], [2k])\}$ and $\sigma_2=\{([0], [0]), ([k], [0]), ([2k-1], 0), ([2k], [0]) \}$.
By Lemma~\ref{lemma:subcomplex_3k-1}, $K$ is homotopy equivalent to $S^2$.

Notice that
\[\vr{T_{3k-1,3k-1}}{k}= K\cup \bigcup_{a\in \Z} L_a^1\cup \bigcup_{b\in \Z} L_b^2.\]

If $a\neq 0$, then $K\cap L_a^1$ is the subcomplex whose facets are $\{\{([a], [b]), ([a], [b+2k-1]), ([a], [b+2k])\}$, $\{([a], [b+k]), ([a], [b+2k-1]), ([a], [b+2k])\}$, $\{([a], [b]), ([a], [b+k])\}\}$, which is homotopic to the circle $S^1$.
Therefore, $K\cup L_a^1$ is homotopy equivalent to $K\vee L_a^1\vee\Sigma(S^1) \simeq S^2\vee S^3\vee S^2$ by Lemma~\ref{cup_simp}.
For $a=0$, $K\cap L_a^1$ is contractible and hence $K\cup L_a^1$ is homotopy equivalent to $S^2\vee S^3$.
The same results hold for $K\cup L_b^2$.
Then an inductive argument yields the first result in the theorem.

\medskip

Now we consider the complex $\vr{T_{5,5}}{2}$.
The Vietoris-Rips complex of the cyclic graph $C_5$ with scale $2$ is generated by a $5$-simplex and hence contractible.
Therefore, when $k=2$, both $L_a^1$ and $L_b^2$ are contractible.
Let $L$ be the complex with vertex set $T_{5, 5}$ and facets $M_{5, 2}$.
Similar as Lemma~\ref{lem:facets_3kor3k-1}, a facet in $\vr{T_{5,5}}{2}$ is either in $M_{5, 2}$ or $L_a^1$, $L_b^2$ for some $a$, or  $b$.
A similar approach as Lemma~\ref{lemma:subcomplex_3k} shows that the complex $L\cup L_0^1\cup L_0^2$ is homotopy equivalent to $S^2$.
And the intersection of each of the $8$ remaining $L_a^1$ or $L_b^2$ with $L\cup L_0^1\cup L_0^2$ is homotopy equivalent to $S^1$.
Hence by induction and Lemma~\ref{cup_simp} the complex $\vr{T_{5,5}}{2}$ is homotopy equivalent to $\bigvee_{9}S^2$.
\end{proof}

\section{Connection between torus grid graphs and cross-polytopes}
\label{sec:cross-polytopes}

In this section we explore the connection between the torus grid graphs and cross-polytopes as noted in Proposition~\ref{Proposition:croo-antipode}.
We state first some notation and results.

An \emph{$n$-dimensional cross-polytope} is a ball in $\mathbb R^n$ with the $\ell_1$ norm, that is, it is a set of the form $\{ x\in \mathbb R^n : \| x - c\|_1 \leq \varepsilon \}$, for some $\varepsilon >0$ and some $c\in \mathbb R^n$.
All $n$-dimensional cross-polytopes are isomorphic equivalent, and hence we only focus on the case where $\varepsilon =1$ and $c = \textbf{0}$ (the origin of $\mathbb R^n$); we refer to this case as ``the'' $n$-dimensional cross-polytope.
Equivalently, if $\{ \hat e_k$ : $1\leq k \leq n\}$ is the standard basis in $\mathbb R^n$, then the $n$-dimensional cross-polytope is the convex hull of the vertex set $\{ \pm \ \hat e_k : 1\leq k \leq n \}$.
The \emph{boundary} of the $n$-dimensional cross-polytope is the set $\{ x\in \mathbb R^n : \|x||_1 = 1 \}$, which is homeomorphic to the unit sphere $S^{n-1}$ in $\mathbb R^n$.

Recall that the diameter of the torus grid graph $T_{n,n}$ is $n$ if $n$ is even, and $n-1$ if $n$ is odd.
For scale $r=\diam(T_{n,n})-1$, the Vietoris--Rips complex $\vr{T_{n,n}}{r}$ will have all possible edges \emph{except} that each vertex will not be connected to its ``antipodes'' (the vertices at distance $r$):

\begin{lemma}
\label{Lemma:VR-antipode}
Let $n$ be a positive integer and $r=\diam(T_{n,n})-1$.
Then in the 1-skeleton of $\vr{T_{n,n}}{r}$, each vertex is connected to all but
\begin{itemize}
    \item exactly one vertex, if $n$ is even,
    \item exactly four vertices, if $n$ is odd.
\end{itemize}
\end{lemma}
\begin{proof}
First note that in the cycle graph $C_n$, given a vertex $u$, there is exactly one antipode of $u$ in $C_n$, if $n$ is even, and exactly two antipodes of $u$ in $C_n$ is $n$ is odd.
This implies that given any vertex $(u,u')$ in $T_{n,n}$ there is exactly one (four) other vertex (vertices) at distance $\diam(T_{n,n})$ from $(u,u')$ if $n$ is even (odd).
\end{proof}

The previous lemma motivates the following definition.
We say that a graph is \emph{antipode} if every vertex is connected to all but exactly one vertex.
For example, the 1-skeleton of $\vr{T_{n,n}}{\diam(T_{n,n})-1}$ is an antipode graph when $n$ is even.
Note that from the definition, the size of the vertex set of an antipode graph must be even.

Antipode graphs are a special case of the \emph{Tur\'an} graphs $T(m, l)$ (see~\cite{turan1954theory}): a vertex set of size $m$ is partitioned into $l$ many subsets, with sizes as equal as possible, and then connecting two vertices by an edge if and only if they belong to different subsets.
Hence, an antipode graph is isomorphic equivalent to $T(2m, m)$ for some $m$.

We recall that the \emph{clique complex} $\cl(G)$ of a graph $G$ is the simplicial complex with vertex set $V(G)$, and with a finite subset $\sigma$ as a simplex $\sigma$ if and only if $\sigma$ induces a complete subgraph of $G$.

\begin{proposition}
\label{Proposition:croo-antipode}
The 1-skeleton of an $n$-dimensional cross-polytope is an antipode graph, in fact of type $T(2n,n)$.
Conversely, the clique complex of an antipode graph of $2n$ vertices is toplogically homeomorphic to the boundary of the $n$-dimensional cross-polytope.
\end{proposition}

\begin{proof}
For the first part, since the $n$-dimensional cross-polytope is the convex hull of the vertex set $\{ \pm \ \hat e_k \in \mathbb R^n~|~1\leq k \leq n \}$, this implies that in its 1-skeleton, each vertex $\hat e_k$ is connected to all vertices except $-\hat e_k$.

Now we prove the second part.
Note that the boundary of the cross-polytope is a simplicial complex that contains a subset $\sigma$ of $\{ \pm \ \hat e_k \in \mathbb R^n~|~1\leq k \leq n \}$ as a simplex if and only if $\{\hat e_k,-\hat e_k\}\nsubseteq \sigma$ for all $k$.
The same is true for the clique complex of an antipode graph with $2n$ vertices (equivalently, for the Tur\'an graph $T(2n,n)$).
Indeed, if we let $\{\pm v_1,\ldots,\pm v_n\}$ be the vertex set of an antipode graph with $2n$ vertices, where each vertex $v_k$ is connected to all other vertices besides $-v_k$, then a subset $\sigma$ of $\{\pm v_1,\ldots,\pm v_n\}$ is a simplex in the clique complex of the antipode graph if and only if $\{v_k,-v_k\}\nsubseteq \sigma$ for all $k$.
So, we can obtain a homeomorphism from the boundary of the $n$-dimesnional cross-polytope to the clique complex of an antipode graph with $2n$ vertices by mapping each vertex $\hat e_k$ to $v_k$, and then extending linearly to simplices.
\end{proof}

\begin{corollary}
\label{cor:cross-polytopal}
Let $n$ be even.
Then $\vr{T_{n,n}}{\diam(T_{n,n})-1} \simeq S^{\frac{n^2}{2} -1}$.
\end{corollary}

\begin{proof}
By Lemma~\ref{Lemma:VR-antipode}, the 1-skeleton of $\vr{T_{n,n}}{\diam(T_{n,n})-1}$ is an antipode graph, indeed of type $T(n^2, \frac{n^2}{2})$.
Also note that $\vr{T_{n,n}}{\diam(T_{n,n})-1}$ coincides with the clique complex of its 1-skeleton.
Hence by Proposition~\ref{Proposition:croo-antipode}, $\vr{T_{n,n}}{\diam(T_{n,n})-1}$ is homeomorphic to the boundary of the $\frac{n^2}{2}$-dimensional cross-polytope, which is homeomorphic to  $S^{\frac{n^2}{2} -1}$.
\end{proof}

\section{The homotopy types of $\vr{T_{5,5}}{3}$ and $\vr{T_{7,7}}{4}$}
\label{sec:homotopy-through-homology}

In this section, we use homology computations (with integer coefficients) and the Hurewicz and Whitehead theorems to determine the homotopy types $\vr{T_{5,5}}{3}\simeq \bigvee_9 S^4$ and $\vr{T_{7,7}}{4}\simeq S^3$.

The homology groups for $\vr{T_{n,n}}{r}$ in Table~\ref{table:Torus-Grids} were computed with $\Z/2\Z$ coefficients, using Ripser software~\cite{bauer2021ripser}.
However, we also computed the homology groups for $\vr{T_{n,n}}{r}$ with $r\le 4$ with $\Z$ coefficients, using Polymake software~\cite{polymake:2000}.
All of the homology groups were free in this range.
So, while entry $i\dv d$ in Table~\ref{table:Torus-Grids} is shorthand for $H_i(\vr{T_{n,n}}{k};\Z/2) \cong (\Z/2)^d$, it is moreover true that if $r\le 4$, then entry $i\dv d$ in Table~\ref{table:Torus-Grids} means $H_i(\vr{T_{n,n}}{k};\Z) \cong \Z^d$ (and all other positive dimensions have trivial homology).
In particular, Polymake computations show that the integral homology of $\vr{T_{5,5}}{3}$ is the same as that of $\bigvee_9 S^4$, and the integral homology of $\vr{T_{7,7}}{4}$ is the same as that of $S^3$.
We now show that these are homotopy equivalences.

Let $Y$ be a simply connected CW complex such that the only nonzero reduced homology group is $\tilde{H}_n(Y)\cong\Z^a$.
It follows from the Hurewicz and Whitehead theorems that we have a homotopy equivalence $Y\simeq \bigvee_a S^n$.
This is given by Proposition~4C.1 of~\cite{Hatcher} and stated clearly in Theorem~1 of~\cite{carnero2024homotopy}; see Theorem~2 of~\cite{carnero2024homotopy} for an interesting generalization.
Since we have used Polymake to determine the integral homology of $\vr{T_{5,5}}{3}$ is the same as that of $\bigvee_9 S^4$, and the integral homology of $\vr{T_{7,7}}{4}$ is the same as that of $S^3$, it suffices to show that these Vietoris--Rips complexes are simply connected.

We prove simple connectedness using the appendix by Barmak in Farber's paper~\cite{farber2023large}, which consolidates prior work~\cite{meshulam2001clique,meshulam2003domination,chudnovsky2000systems,Kahle2009} lower bounding the connectivity of clique complexes of graphs.
Recall that a nonempty topological space $Y$ is \emph{$k$-connected} if the homotopy groups $\pi_i(Y)$ are trivial for all $i\le k$.
Barmak proves that if a simplicial complex is $(2k+2)$-conic (meaning every subcomplex with at most $(2k+2)$ vertices is contained in a simplicial cone), then the complex is $k$-connected~\cite[Appendix, Theorem~4]{farber2023large}.
If any $2k+2$ closed balls of radius $r$ in the metric space $X$ intersect, then each such intersection point is the apex of a simplicial cone in $\vr{X}{r}$ containing any subcomplex on those ball centers, showing that $\vr{X}{r}$ is $(2k+2)$-conic and hence $k$-connected.\footnote{See~\cite[Corollary~4.4]{adams2024connectivity} for the case using \emph{open} Vietoris--Rips complexes with the $\diam(\sigma)<r$ convention (and hence open balls in $X$), instead of the \emph{closed} Vietoris--Rips complexes with the $\diam(\sigma)\le r$ convention (and hence closed balls in $X$) as we use here.}
Taking $k=1$ and $X=T_{n,n}$, we see that if any $4$ closed balls in $T_{n,n}$ of radius $r$ intersect, then $\vr{T_{n,n}}{r}$ is simply connected.

The graph $T_{5,5}$ has $5^2=25$ vertices, and each closed ball of radius $r=3$ in $T_{5,5}$ contains $21$ vertices.
Therefore, the intersection of any $4$ closed balls of radius $3$ in $T_{5,5}$ contains at least $25-4(25-21)=9$ vertices, and is hence nonempty.
This shows $\vr{T_{5,5}}{3}$ is simply connected, and hence $\vr{T_{5,5}}{3}\simeq \bigvee_9 S^4$.

The graph $T_{7,7}$ has $7^2=49$ vertices, and each closed ball of radius $r=4$ in $T_{7,7}$ contains $37$ vertices.
Therefore, the intersection of any $4$ closed balls of radius $4$ in $T_{7,7}$ contains at least $49-4(49-37)=1$ vertex, and is hence nonempty.
This shows $\vr{T_{7,7}}{4}$ is simply connected, and hence $\vr{T_{7,7}}{4}\simeq S^3$.
We have proven the following theorem.

\begin{theorem}
\label{thm:homotopy-through-homology}
There are homotopy equivalences $\vr{T_{5,5}}{3}\simeq \bigvee_9 S^4$ and $\vr{T_{7,7}}{4}\simeq S^3$.
\end{theorem}

\section{Conclusion}
\label{sec:conclusion}

We end with a list of open questions.

\begin{question}
The smallest complexes that we have not proven anything about (see Table~\ref{table:Torus-Grids}) are of the form $\vr{T_{n,n}}{3}$, i.e.\ the Vietoris---Rips complexes of $n\times n$ grids on the torus at scale parameter $3$, for $n=6$ and $7$.
Homology computations show that $\vr{T_{6,6}}{3}$ has $\beta_3=1$ and $\beta_5=12$, and that $\vr{T_{7,7}}{3}$ has $\beta_3=1$ and $\beta_4=14$.
What are the homotopy types of these spaces?
\end{question}

\begin{question} What are the homotopy types of $\vr{T_{6,6}}{4}$, $\vr{T_{7,7}}{4}$, $\vr{T_{8,8}}{4}$, $\vr{T_{9,9}}{4}$, and $\vr{T_{10,10}}{4}$?
In particular, is $\vr{T_{7,7}}{4}$ homotopy equivalent to a $3$-sphere?
\end{question}

\begin{question}
Is it possible to determine in what dimensions $i$ the reduced homology groups $\tilde{H}_i(\vr{T_{n,n}}{k})$ are nonzero?
\end{question}

\begin{question}
For $k\ge 3$, is it the case that $\tilde{H}_i(\vr{T_{3k-2,3k-2}}{k})$ is nonzero if and only if $i=3,4$?

For $k\ge 5$, is 
\[\beta_3(\vr{T_{3k-2,3k-2}}{k})=
\begin{cases}
6k-3&\text{if }k\text{ is odd} \\
6k-1&\text{if }k\text{ is even?}
\end{cases}\]

For $k\ge 5$, is 
\[\beta_4(\vr{T_{3k-2,3k-2}}{k})=
\begin{cases}
6k-4&\text{if }k\text{ is odd} \\
6k-2&\text{if }k\text{ is even?}
\end{cases}\]
\end{question}

\begin{question}
Is it true that for $k\geq 7$, the complex $\vr{T_{3k-3,3k-3}}{k}$ has
\[
\beta_3(\vr{T_{3k-3,3k-3}}{k})=3
\quad\text{and}\quad
\beta_4(\vr{T_{3k-3,3k-3}}{k})=2?
\]
\end{question}

\begin{question}
Is it true that for $k\geq 2$, the complex $\vr{T_{2k, 2k}}{k}$ is homotopy equivalent to the wedge sum of $S^3$ with a $4k$-fold wedge sum of $S^{2k-1}$?
\end{question}

\begin{question}
How can Theorem~\ref{thm:homotopy-through-homology} be proven without relying on integral homology computations?
\end{question}

\begin{question}
\label{ques:3sphere}
We conjecture that $\vr{T_{n,n}}{k}$ is homotopy equivalent to a $3$-sphere for a countable family of $(n,k)$ pairs, definitely including $(n,k)=(7,4)$ (see Theorem~\ref{thm:homotopy-through-homology}), and potentially including a subset of
\begin{align*}
(n,k)\in\ &\{(8,5),(9,5)\}\cup\{(11,6),(12,6)\} \\
&\cup\{(n,7)~|~12\le n\le 15\} \\
&\cup\{(n,8)~|~13\le n\le 19\} \\
&\cup\{(n,9)~|~16\le n\le 21\} \\
&\cup\{(n,10)~|~17\le n\le 24\} \\
&\cup\{(n,11)~|~18\le n\le 26\}.
\end{align*}
We conjecture that these complexes are homotopy equivalent to a 3-sphere $S^3$ that is formed from the hollow torus $S^1 \times S^1$ by gluing in two solid tori (as the scale increases), in order to obtain the $3$-sphere as the standard genus-1 Heegaard decomposition $S^3=(S^1\times D^2)\cup_{S^1\times S^1}(D^2 \times S^1)$.
\end{question}

\begin{question}
For any $n$ and for any $k\le k'$, is the rank of the map $H_3(\vr{T_{n,n}}{k})\to H_3(\vr{T_{n,n}}{k'})$ induced by inclusion equal to
\[\min\{\beta_3(\vr{T_{n,n}}{k}),\beta_3(\vr{T_{n,n}}{k'})\}?\]
In other words, are the $3$-dimensional persistent homology groups as large as they can be (given the $3$-dimensional Betti numbers)?
\end{question}

\begin{question}
We note that $H_4(\vr{T_{10,10}}{4})$ has rank $60$, $H_4(\vr{T_{10,10}}{5})$ has rank $0$, and $H_4(\vr{T_{10,10}}{6})$ has rank $39$.
(As a result, the rank of the map $H_4(\vr{T_{10,10}}{4})\to H_4(\vr{T_{10,10}}{6})$ induced by inclusion is zero.)
When is it possible for the rank of homology to be non-monotonic as the scale increases?
\end{question}

\begin{question}
Is it true that for $k\geq 0$, the complex $\vr{T_{7+5k,7+5k}}{3+2k}$ has \\$\beta_4(\vr{T_{7+5k,7+5k}}{3+2k})=14+10k$?

Tables~\ref{table:Torus-Grids} and~\ref{table:cont} confirm this for $0\le k\le 4$.
\end{question}

\begin{question}
Is it true that for $k\geq 3$, the complex $\vr{T_{5k,5k}}{2k}$ has $\beta_3(\vr{T_{5k,5k}}{2k})=1$ and $\beta_4(\vr{T_{5k,5k}}{2k})=10k^2$?

Tables~\ref{table:Torus-Grids} and~\ref{table:cont} confirm this for $3\le k\le 5$.
\end{question}

\begin{question}
Is it true that $\vr{T_{n,n}}{k}$ is not contractible for
\[
k <
\begin{cases}
n & \text{if }n\text{ is even} \\
n-1 & \text{if }n\text{ is odd?}
\end{cases}
\]
\end{question}

%\begin{question}
%Is $\vr{T_{n,n}}{k}$ always homotopy equivalent to either a torus or to a wedge of spheres?
%\end{question}

\section{Acknowledgments}
We thank Micha{\l} Adamaszek for help with an earlier version of Table~\ref{table:Torus-Grids}, and \v{Z}iga Virk for an idea related to Theorem~\ref{thm:S2S3}.

\bibliographystyle{plain}
\bibliography{references}

\appendix
\section{}

The following result about maximal simplices is referenced in Section~\ref{ssec:maximal}.

\begin{lemma}
\label{lemma:appendix}
We have
\[
M(\vr{C_{n}}{k}) = 
\begin{cases}
\{\sigma_i: i=0, \ldots, n-1\} &\text{for }n>3k, \\
\{\sigma_i:i=0, \ldots, n-1\} \cup \{\sigma_i': i=0, \ldots, n-1\} &\text{for }n=3k, k\geq 2, \\
\{\sigma_i:i=0, \ldots, n-1\} \cup \{\sigma_i'': i=0, \ldots, n-1\} &\text{for }n=3k-1, k\geq 3,
\end{cases}
\]
where  $\sigma_i=\{v_i, v_{i+1},\ldots, v_{i+k}\}$,  $\sigma_i'=\{v_i, v_{i+k}, v_{i+2k}\}$,  and $\sigma_i''=\{v_i, v_{i+k}, v_{i+2k-1}, v_{i+2k}\}$, with all indices taken modulo $n$.
\end{lemma}

\begin{proof}
When $n>3k$, it is clear that the simplices of type $\sigma_i$ are maximal.
So, we need only to verify that there are no other facets.

Let $\sigma\in M(\vr{C_n}{k})$.
We define the \emph{arc} of a vertex $v_i$ in $C_n$ as the set of $k+1$ consecutive vertices starting at $v_i$, i.e., $\{v_i,v_{i+1},\ldots, v_{i+k}\}$.
Fix a vertex $v_j\in\sigma$, and note that all vertices in $\sigma$ must lie in $B_{C_n}[v_j,k]$.
In this proof, for convenience, we use $B[v,k]$ to denote $B_{C_n}[v,k]$ for any vertex $v\in C_n$ and $k\geq 0$.
Let $v_t\in\sigma$ be such that $d(v_t, v_{j-k})=\min\{d(v_i,v_{j-k}): v_i\in\sigma\}$.
We claim that every vertex in $\sigma$ must be contained in the arc of $v_t$.
To verify this, we have two cases to check: $v_t=v_j$ and $v_t\neq v_j$.

\begin{figure}[h]
\begin{center}
\begin{tikzpicture}
\def \n {10} 
\def \radius {1.6cm} 

\foreach \s in {1,...,\n}
{
  \node[draw, circle, fill=black, inner sep=2pt, minimum size=8pt] 
    (v\s) at ({360/\n * (\s - 1)}:\radius) {};
  
  \foreach \s in {5,6,7,8,9,10,1}{
  \node[draw, circle, fill=red, inner sep=2pt, minimum size=8pt] 
  at ({360/\n * (\s- 1)}:\radius) {};
  }

  %labels
  \node[font=\small] at ({360/\n * (\s - 1)}:\radius + 0.5cm) {$v_{\s}$};
 }
 \draw[red, thick] ({360/\n * (8 - 1)}:\radius) circle (0.27cm);
 \draw[blue, thick] ({360/\n * (6 - 1)}:\radius) circle (0.27cm);
 %edges
 \foreach \s in {1,...,\n}
 {
  \pgfmathtruncatemacro{\next}{mod(\s, \n) + 1}
  \draw (v\s) -- (v\next);
 }

\draw[red, thick] (v8) -- (v7);
\draw[red, thick] (v8) -- (v9);
\draw[red, thick] (v9) -- (v10);
\draw[red, thick] (v10) -- (v1);
\draw[red, thick] (v6) -- (v7);
\draw[red, thick] (v5) -- (v6);

\begin{scope}[xshift=5.5cm]
\def \n {10} 
\def \radius {1.6cm} 

\foreach \s in {1,...,\n}
{
  \node[draw, circle, fill=black, inner sep=2pt, minimum size=8pt] 
    (v\s) at ({360/\n * (\s - 1)}:\radius) {};

  %labels
  \node[font=\small] at ({360/\n * (\s - 1)}:\radius + 0.5cm) {$v_{\s}$};
 }
 \draw[blue, thick] ({360/\n * (6 - 1)}:\radius) circle (0.27cm);
 %edges
 \foreach \s in {1,...,\n}
 {
  \pgfmathtruncatemacro{\next}{mod(\s, \n) + 1}
  \draw (v\s) -- (v\next);
 }

   \foreach \s in {6,7,8,9}{
  \node[draw, circle, fill=blue, inner sep=2pt, minimum size=8pt] 
  at ({360/\n * (\s- 1)}:\radius) {};
  }

\draw[blue, thick] (v6) -- (v7);
\draw[blue, thick] (v7) -- (v8);
\draw[blue, thick] (v8) -- (v9);

\end{scope}

\end{tikzpicture}
\caption{Visualization of $n=10$ case where $v_6$ is the vertex in $\sigma$ closest to $v_{j-k}=v_{8-3}=v_5$.
The left diagram highlights $B[v_8,3]$ in red, while the right diagram displays the arc of $v_6$ in blue.}
\end{center}
\end{figure}

Suppose $v_t=v_j$.
For the sake of contradiction, assume that $\sigma$ contains a vertex $v_s$ not in the arc of $v_t$.
Then $v_s\in B[v_j,k]$ and $v_s\not\in\{v_{j-k}, v_{j-k+1},\ldots, v_{j-1}\}$ as $v_t$ is the closest point in $\sigma$ to $v_{j-k}$.
Hence, $v_s\not\in B[v_j,k]$ as $v_s\not\in\{v_{j-k}, v_{j-k+1},\ldots, v_{j-1}\}\cup\{v_j,v_{j+1},\ldots, v_{j+k}\}$, contradicting that $v_s\in B[v_j,k]$.
Hence, $\sigma$ must be of the form $\sigma_i$.

Suppose $v_t\neq v_j$.
It is easy to see that $v_t\in\{v_{j-k},\ldots, v_{j-1}\}$ as otherwise $d(v_t, v_{j-k})>d(v_j, v_{j-k})=k$.
Again assume that $\sigma$ contains a vertex $v_s$ not in the arc of $v_t$.
Since $v_t$ is the vertex in $\sigma$ closest to $v_{j-k}$, then $v_s\in B[v_j,k]\backslash\{v_{j-k}, v_{j-k+1}, \ldots, v_{t+k}\}$.
We claim that $d(v_t,v_s)>k$.
Note that $d(v_s,v_t)$ is the minimum of the distances from $v_s$ to $v_t$ measured in the clockwise and counterclockwise directions.
The counterclockwise distance passing through $v_j$ is clearly greater than $k$ since $v_s$ is not in the arc of $v_t$.
Moreover, the clockwise distance is greater than $k$ as the distance between the endpoints of $B[v_j, k]$ is minimally $k+1$.
Thus, in either case $d(v_t, v_s)>k$, contradicting that $v_s\in\sigma$.
Consequently, all $\sigma\in M(\vr{C_n}{k})$ must be of the form $\sigma_i$.

\begin{figure}[h]
\begin{center}
\begin{tikzpicture}
\def \n {10} 
\def \radius {1.6cm} 
\def \arcshift {0.8cm}

\pgfmathsetmacro{\startAngle}{360/\n * (1 - 1)}
\pgfmathsetmacro{\endAngle}{360/\n * (6 - 1)}

\foreach \s in {1,...,\n}
{
  \node[draw, circle, fill=black, inner sep=2pt, minimum size=8pt] 
    (v\s) at ({360/\n * (\s - 1)}:\radius) {};

  \node[font=\small] at ({360/\n * (\s - 1)}:\radius + 0.5cm) {$v_{\s}$};
}

\foreach \s in {5,10,1}
{
  \node[draw, circle, fill=red, inner sep=2pt, minimum size=8pt] 
    at ({360/\n * (\s - 1)}:\radius) {};
}

\foreach \s in {6,7,8,9}
{
  \node[draw, circle, fill=blue, inner sep=2pt, minimum size=8pt] 
    at ({360/\n * (\s - 1)}:\radius) {};
}

\draw[green, thick] ({360/\n * (10 - 1)}:\radius) circle (0.27cm);
\draw[green, thick] ({360/\n * (1 - 1)}:\radius) circle (0.27cm);

\foreach \s in {1,...,\n}
{
  \pgfmathtruncatemacro{\next}{mod(\s, \n) + 1}
  \draw (v\s) -- (v\next);
}

\draw[red, thick] (v9) -- (v10);
\draw[red, thick] (v10) -- (v1);
\draw[red, thick] (v5) -- (v6);

\draw[blue, thick] (v6) -- (v7);
\draw[blue, thick] (v7) -- (v8);
\draw[blue, thick] (v8) -- (v9);

\def \arcshift {0.8cm} %move arc away from circle

 %midpoint computation
\pgfmathsetmacro{\midAngle}{(\startAngle+\endAngle)/2}

%draw arc
\draw[thick,|-|]
    ({\endAngle}:\radius + \arcshift) 
    arc[start angle={\endAngle}, 
        end angle={\startAngle}, 
        x radius={\radius + \arcshift}, 
        y radius={\radius + \arcshift}];

\node[font=\tiny, above, rotate={\midAngle-90}] 
    at ({\midAngle}:{\radius + \arcshift + 0.02cm}) 
    {$k+2$};

\begin{scope}[xshift=6cm]
\def \n {10} 
\def \radius {1.6cm} 
\def \arcshift {0.8cm}

\pgfmathsetmacro{\startAngle}{360/\n * (6 - 1)}
\pgfmathsetmacro{\endAngle}{360/\n * (10 - 1)}

\foreach \s in {1,...,\n}
{
  \node[draw, circle, fill=black, inner sep=2pt, minimum size=8pt] 
    (v\s) at ({360/\n * (\s - 1)}:\radius) {};

  \node[font=\small] at ({360/\n * (\s - 1)}:\radius + 0.5cm) {$v_{\s}$};
}

\foreach \s in {5,10,1}
{
  \node[draw, circle, fill=red, inner sep=2pt, minimum size=8pt] 
    at ({360/\n * (\s - 1)}:\radius) {};
}

\foreach \s in {6,7,8,9}
{
  \node[draw, circle, fill=blue, inner sep=2pt, minimum size=8pt] 
    at ({360/\n * (\s - 1)}:\radius) {};
}

\draw[green, thick] ({360/\n * (10 - 1)}:\radius) circle (0.27cm);
\draw[green, thick] ({360/\n * (1 - 1)}:\radius) circle (0.27cm);

\foreach \s in {1,...,\n}
{
  \pgfmathtruncatemacro{\next}{mod(\s, \n) + 1}
  \draw (v\s) -- (v\next);
}

\draw[red, thick] (v9) -- (v10);
\draw[red, thick] (v10) -- (v1);
\draw[red, thick] (v5) -- (v6);

\draw[blue, thick] (v6) -- (v7);
\draw[blue, thick] (v7) -- (v8);
\draw[blue, thick] (v8) -- (v9);

\def \arcshift {0.8cm} %move arc away from circle

 %midpoint computation
\pgfmathsetmacro{\midAngle}{(\startAngle+\endAngle)/2}

%draw arc
\draw[thick,|-|]
    ({\endAngle}:\radius + \arcshift) 
    arc[start angle={\endAngle}, 
        end angle={\startAngle}, 
        x radius={\radius + \arcshift}, 
        y radius={\radius + \arcshift}];

\node[font=\tiny, above, rotate={\midAngle+90}] 
    at ({\midAngle}:{\radius + \arcshift + 0.5cm}) 
    {$k+1$};

\end{scope}

\end{tikzpicture}
\caption{Continuation of the above case, illustrating the minimal distances in both the clockwise and counterclockwise directions from $v_6$ to vertices in $B[v_8,3]$ that are not in the arc of $v_6$.
These vertices are highlighted with green outlines for clarity.}
\end{center}
\end{figure}

When $n=3k$, there are two types of facets.
Denote $\sigma_i'=\{v_i, v_{i+k}, v_{i+2k}\}$, with all indices taken modulo $n$.
We claim
$M(\vr{C_{n}}{k})=\{\sigma_i:i=0, 1,\ldots, n-1\} \cup \{\sigma_i': i=0, 1,\ldots, n-1\}$.

Once again, it is evident that each $\sigma_i$ is maximal in this case.

We next show that each $\sigma_i'$ is maximal.
Let $\sigma_i'=\{v_i, v_{i+k}, v_{i+2k}\}$.
Suppose that $\sigma_i'$ is not maximal.
Then it must be contained in a larger simplex, say $\tau$.
Let $v_j\in \tau\backslash\sigma_i'$.
Then $v_j$ is between some pair of vertices in $\sigma_i'$.
Without loss of generality, suppose $v_j$ is contained in the arc between $v_i$ and $v_{i+k}$.
It is clear that $d(v_j,v_{i+2k})\geq k+1$, contradicting that $\tau\in \vr{C_n}{k}$.
Hence, the simplices of the form $\sigma_i'$ are maximal.

Therefore, it suffices to show that any facet not of the form $\sigma_i$ must be of the $\sigma_i'$ type.
Let $\sigma\in M(\vr{C_n}{k})$, and fix a vertex $v_j\in\sigma$.
Next, choose $v_t\in\sigma$ such that $d(v_t, v_{j-k})=\min\{d(v_i,v_{j-k}): v_i\in\sigma\}$.
We claim that if $d(v_{t}, v_j)<k$, then $\sigma$ is of the form $\sigma_i$.
Otherwise, if $d(v_{t}, v_j)=k$, $\sigma$ is either of type $\sigma_i'$ or $\sigma_i$.

For the first assertion, if $0<d(v_{t}, v_j)<k$, the argument is identical to the $n>3k$ case.
If instead $d(v_{t}, v_j)=0$, then $\sigma\subseteq\{v_j, v_{j+1},\ldots, v_{j+k}\}$, and since $\sigma$ is maximal, $\sigma$ must be of the form $\sigma_i$.

For the second assertion, suppose $d(v_{t}, v_j)=k$.
We have the following two cases to check: $v_t=v_{j-k}$ or $v_t=v_{j+k}$.

If $v_t=v_{j+k}$, then $k=\min\{d(v_i,v_{j-k}): v_i\in\sigma\}$.
This forces $\sigma$ to be of the form $\sigma_i$ as $\sigma\subseteq \{v_j, v_{j+1},\ldots, v_{j+k}\}$.

Suppose $\sigma$ is not of the form $\sigma_i$, and let $v_s\in\sigma\backslash\{v_t, v_j\}$.
If $v_t=v_{j-k}$, then we claim $v_s\in\{v_{j+1},\ldots, v_{j+k}\}$.
Suppose instead $v_s\in\{v_{j-k+1},\ldots, v_{j-1}\}$.
If $k=2$, then $v_s=v_{j-1}$, which yields $\sigma=\{v_{j-2},v_{j-1}, v_j\}$, a contradiction.
If $k\geq 3$, then since $\sigma$ is not of the form $\sigma_i$, simplex $\sigma$ must contain some $v_i\in \{v_{j+1},\ldots, v_{j+k}\}$.
The only $v_i\in \{v_{j+1},\ldots, v_{j+k}\}$ that is within $k$ of both $v_j$ and $v_{j-k}$ is $v_{j+k}$.
However, $d(v_s, v_{j+k})>k$, a contradiction.
Hence, $v_s\in \{v_{j+1},\ldots, v_{j+k}\}$.
Moreover, it is clear that $v_s=v_{j+k}$ as otherwise $d(v_{j-k}, v_s)>k$.
Hence, $\sigma$ is of type $\sigma_i'$.

\begin{figure}[h]
\begin{center}
\begin{tikzpicture}
\def \n {9} 
\def \radius {1.6cm} 
\def \arcshift {0.8cm}

\pgfmathsetmacro{\startAngle}{360/\n * (1 - 1)}
\pgfmathsetmacro{\endAngle}{360/\n * (6 - 1)}

\foreach \s in {1,...,\n}
{
  \node[draw, circle, fill=black, inner sep=2pt, minimum size=8pt] 
    (v\s) at ({360/\n * (\s - 1)}:\radius) {};

  \node[font=\small] at ({360/\n * (\s - 1)}:\radius + 0.5cm) {$v_{\s}$};
}

\foreach \s in {4,5,6,7,8,9,1}
{
  \node[draw, circle, fill=red, inner sep=2pt, minimum size=8pt] 
    at ({360/\n * (\s - 1)}:\radius) {};
}

\draw[red, thick] ({360/\n * (7 - 1)}:\radius) circle (0.27cm);
\draw[red, thick] ({360/\n * (4 - 1)}:\radius) circle (0.27cm);
\draw[green, thick] ({360/\n * (8 - 1)}:\radius) circle (0.27cm);
\draw[green, thick] ({360/\n * (9 - 1)}:\radius) circle (0.27cm);
\draw[green, thick] ({360/\n * (0)}:\radius) circle (0.27cm);

\foreach \s in {1,...,\n}
{
  \pgfmathtruncatemacro{\next}{mod(\s, \n) + 1}
  \draw (v\s) -- (v\next);
}

\draw[red, thick] (v9) -- (v1);
\draw[red, thick] (v5) -- (v6);
\draw[red, thick] (v6) -- (v7);
\draw[red, thick] (v7) -- (v8);
\draw[red, thick] (v8) -- (v9);
\draw[red, thick] (v5) -- (v4);

\begin{scope}[xshift=6cm]
\def \n {9} 
\def \radius {1.6cm} 
\def \arcshift {0.8cm}

\pgfmathsetmacro{\startAngle}{360/\n * (1 - 1)}
\pgfmathsetmacro{\endAngle}{360/\n * (4 - 1)}

\foreach \s in {1,...,\n}
{
  \node[draw, circle, fill=black, inner sep=2pt, minimum size=8pt] 
    (v\s) at ({360/\n * (\s - 1)}:\radius) {};

  \node[font=\small] at ({360/\n * (\s - 1)}:\radius + 0.5cm) {$v_{\s}$};
}

\foreach \s in {4,5,6,7,8,9,1}
{
  \node[draw, circle, fill=red, inner sep=2pt, minimum size=8pt] 
    at ({360/\n * (\s - 1)}:\radius) {};
}

\draw[red, thick] ({360/\n * (7 - 1)}:\radius) circle (0.27cm);
\draw[red, thick] ({360/\n * (4 - 1)}:\radius) circle (0.27cm);
\draw[red, thick] ({360/\n * (0)}:\radius) circle (0.27cm);

\foreach \s in {1,...,\n}
{
  \pgfmathtruncatemacro{\next}{mod(\s, \n) + 1}
  \draw (v\s) -- (v\next);
}

\draw[red, thick] (v9) -- (v1);
\draw[red, thick] (v5) -- (v6);
\draw[red, thick] (v6) -- (v7);
\draw[red, thick] (v7) -- (v8);
\draw[red, thick] (v8) -- (v9);
\draw[red, thick] (v5) -- (v4);

\def \arcshift {0.8cm} %move arc away from circle

 %midpoint computation
\pgfmathsetmacro{\midAngle}{(\startAngle+\endAngle)/2}

%draw arc
\draw[thick,|-|]
    ({\endAngle}:\radius + \arcshift) 
    arc[start angle={\endAngle}, 
        end angle={\startAngle}, 
        x radius={\radius + \arcshift}, 
        y radius={\radius + \arcshift}];

\node[font=\tiny, above, rotate={\midAngle-90}] 
    at ({\midAngle}:{\radius + \arcshift + 0.02cm}) 
    {$k$};

\draw[blue, thick] (v7) -- (v1);
\draw[blue, thick] (v1) -- (v4);
\draw[blue, thick] (v7) -- (v4);

\end{scope}

\end{tikzpicture}
\caption{Visualization of $n=9$ case where $v_7$ corresponds to $v_j$ and $v_4$ corresponds to $v_{j-k}$.
The left diagram highlights $B[v_7,3]$ in red, along with all potential $v_s\in\{v_{j+1}, \ldots, v_{j+k}\}$, which are highlighted in green.
The right diagram illustrates that $v_1$, which is $v_{j+k}$ in this case, is the only vertex within $k$ of $v_4$.}
\label{fig:9casekball}
\end{center}
\end{figure}

%\note{If desired, the proof of the $n=3k-1$ case could either be expanded on or entirely removed as it follows a similar argument to the $n=3k$ case.
%I am fine with either approach; let me know what you all prefer.}
When $n=3k-1$, there are two types of facets.
Denote $\sigma_i''=\{v_i, v_{i+k}, v_{i+2k-1}, v_{i+2k}\}$, with all indices taken modulo $n$.
We claim
$M(\vr{C_{n}}{k})=\{\sigma_i:i=0, 1,\ldots, n-1\} \cup \{\sigma_i'': i=0, 1,\ldots, n-1\}$.

The reasoning for this case follows a similar approach to that outlined in the $n=3k$ case.
So, we only show that the simplices of type $\sigma_i''$ are maximal.

Suppose that some $\sigma_i''$ is not maximal.
Then $\sigma_i''$ must be contained in a larger simplex, say $\tau$.
Let $v_j\in \tau\backslash \sigma_i''$.
Note that $v_j$ must be on an arc between one of the following pairs: $\{v_i,v_{i+k}\}, \{v_{i+k}, v_{i+2k-1}\}$, or $\{v_i,v_{i+2k}\}$.
If $v_j$ is on the arc between $\{v_{i+k}, v_{i+2k-1}\}$, then $d(v_j, v_i)\geq k+1$.
Similarly, if $v_j$ is between $\{v_i,v_{i+2k}\}$, then $d(v_j,v_{i+2k-1})\geq k+1$.
Finally, if $v_j$ is a vertex on the arc between $\{v_i,v_{i+k}\}$, we argue that either $d(v_j, v_{i+2k})>k$ or $d(v_j, v_{i+2k-1})>k$.

Suppose that $d(v_j, v_{i+2k})\leq k$.
Since $d(v_i,v_{i+2k})=k-1$, $d(v_j, v_{i+2k})=k$.
Consequently, $d(v_j, v_{i+2k-1})=k+1$, which contradicts the assumption that $\tau\in\vr{C_n}{k}$.
Therefore, the simplices of the form $\sigma_i''$ are maximal in this case.
\end{proof}

\end{document}